\documentclass{lms}
\usepackage{amsmath}
\usepackage{amsfonts}
\usepackage{amssymb}
\usepackage{subfigure}
\usepackage{graphicx}
\usepackage{color}
\usepackage{lineno}
\usepackage{pinlabel}
\usepackage{hyperref}
\newtheorem{theorem}{Theorem}[section] % 1st argument is your name for it
\newtheorem{lemma}[theorem]{Lemma}     % 2nd argument is what is printed
\newtheorem{corollary}[theorem]{Corollary}
\newtheorem{proposition}[theorem]{Proposition}
\newnumbered{assertion}[theorem]{Assertion}    % 1st argument is your name for it
\newnumbered{conjecture}[theorem]{Conjecture}  % 2nd argument is what is printed
\newnumbered{definition}[theorem]{Definition}
\newnumbered{hypothesis}[theorem]{Hypothesis}
\newnumbered{remark}[theorem]{Remark}
\newnumbered{note}[theorem]{Note}
\newnumbered{observation}[theorem]{Observation}
\newnumbered{problem}[theorem]{Problem}
\newnumbered{question}[theorem]{Question}
\newnumbered{algorithm}[theorem]{Algorithm}
\newnumbered{example}[theorem]{Example}
\newunnumbered{notation}[theorem]{Notation}
\newnumbered{axiom}[theorem]{Axiom}
\newcommand{\F}{\mathcal{F}}
\newcommand{\ba}{\backslash}
\newcommand{\btu}{\bigtriangleup}

\title[Embedded graphs and delta-matroids]% end with percent
 {On the interplay between embedded graphs and delta-matroids} % This is the full title of the paper

\author[C. Chun, I. Moffatt, S. D. Noble and R. Rueckriemen]{Carolyn Chun, Iain Moffatt, Steven D. Noble and Ralf Rueckriemen}

\classno{05B35 (primary), 05C10, 05C31, 05C83(secondary).}

\extraline{Ralf Rueckriemen was financed by the DFG through grant RU 1731/1-1.}

\begin{document}
\maketitle

%\linenumbers

\begin{abstract}
The mutually enriching relationship between graphs and matroids has motivated discoveries in both fields. In this paper, we exploit the similar relationship between embedded graphs and delta-matroids. There are well-known connections between geometric duals of plane graphs and duals of matroids. We obtain analogous connections for various types of duality in the literature for graphs in surfaces of higher genus and delta-matroids. Using this interplay, we establish a rough structure theorem for delta-matroids that are twists of matroids, we translate Petrie duality on ribbon graphs to loop complementation on delta-matroids, and we prove that ribbon graph polynomials, such as the Penrose polynomial, the characteristic polynomial, and the transition polynomial, are in fact delta-matroidal.  We also express the Penrose polynomial as a sum of characteristic polynomials.
\end{abstract}

\maketitle

%\tableofcontents

\section{Overview}
Graph theory and matroid theory are mutually enriching. As reported in~\cite{Oxley01}, Tutte famously observed that, ``If a theorem about graphs can be expressed in terms of edges and circuits alone it probably exemplifies a more general theorem about matroids''. In~\cite{CMNR} we proposed that a similar claim holds true for topological graph theory and delta-matroid theory, namely that, ``If a theorem about embedded graphs can be expressed in terms of its spanning quasi-trees then it probably exemplifies a more general theorem about delta-matroids''. In that paper we provided evidence for this claim by showing that, just as with graph and matroid theory, many fundamental definitions and results in topological graph theory and delta-matroid theory are compatible with each other (in the sense that they canonically translate from one setting to the other).
A significant consequence of this connection is that the geometric ideas of topological graph theory provide insight and intuition into the structure of delta-matroids, thus pushing forward the development of both areas.
Here we provide further support for our claim above by presenting results for delta-matroids that are inspired by recent research on ribbon graphs.

We are principally concerned with duality, which for delta-matroids is a much richer and more varied notion than for matroids.
The concepts of duality for plane and planar graphs, and for graphic matroids are intimately connected: the dual of the cycle matroid of an embedded graph corresponds to the matroid of the dual graph (i.e., $M(G)^*=M(G^*)$) if and only if the graph is plane. Moreover, the dual of the cycle matroid of a graph is graphic if and only if the graph is planar. The purpose of this paper is to extend these fundamental graph duality--matroid duality relationships from graphs in the plane to graphs embedded on higher genus surfaces. To achieve this requires us to move from matroids to the more general setting of delta-matroids.

Moving beyond plane and planar graphs opens the door to the various notions of the ``dual'' of an embedded graph that appear in the topological graph theory literature. Here we consider the examples of Petrie duals and direct derivatives~\cite{Wil79}, partial duals~\cite{Ch1} and twisted duals~\cite{EMM}. We will see that these duals are compatible with existing constructions in delta-matroid theory, including twists~\cite{ab1} and loop complementation~\cite{BH11}. We take advantage of the geometrical insights provided by topological graph theory to deduce and prove new structural results for delta-matroids and for their polynomial invariants.

Much of the very recent work on delta-matroids appears in a series of papers by Brijder and Hoogeboom~\cite{BH11,BH13,BHpre2,BH12}, who were originally motivated by an application to gene assembly in one-cell organisms known as cilliates. Their study of the effect of the principal pivot transform on symmetric binary matrices led them to the study of binary delta-matroids. As we will see, the fundamental connections made possible by the abstraction to delta-matroids allows us to view notions of duality in the setting of symmetric binary matrices and the apparently unconnected setting of ribbon graphs as exactly the same thing.

The structure of this paper is as follows. We begin by reviewing delta-matroids, embedded graphs and their various types of duality, and the connection between delta-matroids and embedded graphs. Here we shall describe embedded graphs as ribbon graphs (ribbon graphs are equivalent to graphs cellularly embedded in surfaces). In Section~\ref{3rd} we use the geometric perspectives offered by topological graph theory to present a rough structure theorem for the class of delta-matroids that are twists of matroids. We give some applications to Eulerian matroids, extending a result of Welsh~\cite{We69}.
In Section~\ref{4th}, we show that Petrie duality can be seen as the analogue of a more general delta-matroid operation, namely loop complementation. We show that a group action on delta-matroids due to Brijder and Hoogeboom~\cite{BH11} is the analogue of twisted duality from Ellis-Monaghan and Moffatt~\cite{EMM}. We apply the insights provided by this connection to give a number of structural results about delta-matroids.
In Section~\ref{5th} we apply our results to graph and matroid polynomials. We show that the Penrose polynomial and transition polynomial~\cite{Ai97,EMM11a,Ja90,Pen71} are delta-matroidal, in the sense that they are determined (up to a simple pre-factor) by the delta-matroid of the underlying ribbon graph, and are compatible with Brijder and Hoogeboom's Penrose and transition polynomials of~\cite{BH13}. We relate the Bollob\'as--Riordan and Penrose polynomials to the transition polynomial and find recursive definitions of these polynomials. Finally, we give a surprising expression for the Penrose polynomial of a vf-safe delta-matroid in terms of the characteristic polynomial.

Throughout the paper we emphasise the interaction and compatibility between delta-matroids and ribbon graphs. We provide evidence that this new perspective offered by topological graph theory enables significant advances to be made in the theory of delta-matroids.

\section{Background on delta-matroids}

\subsection{Delta-matroids}

A \emph{set system} is a pair $D=(E,{\mathcal{F}})$ where $E$ is a non-empty finite set, which we call the \emph{ground set}, and $\mathcal{F}$ is a collection of subsets of $E$, called \emph{feasible sets}.
We define $E(D)$ to be $E$ and $\mathcal{F}(D)$ to be $\mathcal{F}$.
A set system $(E,{\mathcal{F}})$ is \emph{proper} if $\mathcal{F}$ is not empty; it is \emph{trivial} if $E$ is empty. For sets $X$ and $Y$,
their \emph{symmetric difference} is denoted by $X\bigtriangleup Y$ and is defined to be $(X\cup Y)-(X\cap Y)$.
Throughout this paper, we will often omit the set brackets in the case of a single element set.
For example, we write $E-e$ instead of $E-\{e\}$, or $F\cup e$ instead of $F\cup \{e\}$.

A \emph{delta-matroid} is a proper set system $D=(E,{\mathcal{F}})$ that satisfies the Symmetric Exchange Axiom:
\begin{axiom}[Symmetric Exchange Axiom]
\label{sea}
 For all $(X,Y,u)$ with $X,Y\in \mathcal{F}$ and $u\in X\bigtriangleup Y$, there is an element $v\in X\bigtriangleup Y$ such that $X\bigtriangleup \{u,v\}$ is in $\mathcal{F}$.
\end{axiom}

Note that we allow $v=u$ in the Symmetric Exchange Axiom.
These structures were first studied by Bouchet in~\cite{ab1}.
If all of the feasible sets of a delta-matroid are equicardinal, then the delta-matroid is a \emph{matroid} and we refer to its feasible sets as its \emph{bases}.
If a set system forms a matroid $M$, then we usually denote $M$ by $(E,\mathcal{B})$, and define $E(M)$ to be $E$ and $\mathcal{B}(M)$ to be $\mathcal{B}$, the collection of bases of $M$. Every subset of every basis is an \emph{independent set}.
For a set $A\subseteq E(M)$, the \emph{rank of $A$}, written $r_M(A)$, or simply $r(A)$ when the matroid is clear, is the size of the largest intersection of $A$ with a basis of $M$.

For a delta-matroid $D=(E,\mathcal{F})$, let $\mathcal{F}_{\max}(D)$ and $\mathcal{F}_{\min}(D)$ be the set of feasible sets with maximum and minimum cardinality, respectively. We will usually omit $D$ when the context is clear.
Let $D_{\max}:=(E,\mathcal{F}_{\max})$ and let $D_{\min}:=(E,\mathcal{F}_{\min})$.
Then $D_{\max}$ is the \emph{upper matroid} and $D_{\min}$ is the \emph{lower matroid} for $D$. These were defined by Bouchet in~\cite{ab2}.
It is straightforward to show that the upper matroid and the lower matroid are indeed matroids.
If the sizes of the feasible sets of a delta-matroid all have the same parity, then we say that it is \emph{even}.

For a proper set system $D=(E,\mathcal{F})$, and $e\in E$, if $e$ is in every feasible set of $D$, then we say that $e$ is a \emph{coloop of $D$}.
If $e$ is in no feasible set of $D$, then we say that $e$ is a \emph{loop of $D$}.
 If $e$ is not a coloop, then we define $D$ \emph{delete} $e$, written $D\ba e$, to be $(E-e, \{F : F\in \mathcal{F}\text{ and } F\subseteq E-e\})$.
 If $e$ is not a loop, then we define $D$ \emph{contract} $e$, written $D/e$, to be $(E-e, \{F-e : F\in \mathcal{F}\text{ and } e\in F\})$.
If $e$ is a
loop or a coloop, then one of $D\ba e$ and $D/ e$ has already been
defined, so we can set $D/e=D\ba e$.
If $D$ is a delta-matroid then both $D\ba e$ and $D/e$ are delta-matroids (see~\cite{BD91}).
If $D'$ is a delta-matroid obtained from $D$ by a sequence of deletions and contractions, then $D'$ is independent of the order of the deletions and contractions used in its construction (see~\cite{BD91}). Any delta-matroid obtained from $D$ in such a way is called a \emph{minor} of $D$.
If $D'$ is a minor of $D$ formed by deleting the elements of $X$ and contracting the elements of $Y$ then we write
$D'=D\setminus X/Y$.

Twists are one of the fundamental operations of delta-matroid theory.
Let $D=(E,{\mathcal{F}})$ be a set system.
For $A\subseteq E$, the \emph{twist} of $D$ with respect to $A$, denoted by $D* A$, is given by $(E,\{A\bigtriangleup X: X\in \mathcal{F}\})$.
The \emph{dual} of $D$, written $D^*$, is equal to $D*E$.
It follows easily from the identity $(F'_1\btu A)\btu(F'_2\btu A)=F'_1\btu F'_2$ that the twist of a delta-matroid is also a delta-matroid, as Bouchet showed in~\cite{ab1}.
However, if $D$ is a matroid, then $D*A$ need not be a matroid.
Note that a coloop or loop of $D$ is a loop or coloop, respectively, of $D^*$. Furthermore if $e\in E$, then $D/e=(D*e)\setminus e$ and $D\setminus e=(D*e)/e$.

For delta-matroids (or matroids) $D_1=(E_1,\mathcal{F}_1)$ and $D_2=(E_2,\mathcal{F}_2)$, where $E_1$ is disjoint from $E_2$, the \emph{direct sum of $D_1$ and $D_2$}, defined in~\cite{geelen} and written $D_1\oplus D_2$, is constructed by
\[D_1\oplus D_2:=(E_1\cup E_2,\{F_1\cup F_2: F_1\in \mathcal{F}_1\text{ and } F_2\in\mathcal{F}_2\}).\]
If $D=D_1\oplus D_2$, for some non-trivial $D_1$ and $D_2$, we say that $D$ is \emph{disconnected} and that $E(D_1)$ and $E(D_2)$ are \emph{separating}. A delta-matroid is \emph{connected} if it is not disconnected.

Binary delta-matroids form an important class of delta-matroids.
For a finite set $E$, let $C$ be a symmetric $|E|$ by $|E|$ matrix over $\mathrm{GF}(2)$, with rows and columns indexed, in the same order, by the elements of $E$.
Let $C\left[ A\right]$ be the principal submatrix of $C$ induced by the set $A\subseteq E$.
We define the delta-matroid $D(C)=(E,\F)$, where $A\in \F$ if and only if $C[A]$ is non-singular over $\mathrm{GF}(2)$. By convention $C[\emptyset]$ is non-singular.
Bouchet showed in~\cite{abrep} that $D(C)$ is indeed a delta-matroid.
A delta-matroid is said to be \emph{binary} if it has a twist that is isomorphic to $D(C)$ for some symmetric matrix $C$ over $\mathrm{GF}(2)$.

When applied to matroids, this definition of binary agrees with the usual definition of a binary matroid (see~\cite{abrep}).

\subsection{Ribbon graphs}\label{dbh}

A \emph{ribbon graph} $G =\left(V(G),E(G)\right)$ is a surface with boundary, represented as the union of two sets of discs: a set $V (G)$ of \emph{vertices} and a set of \emph{edges} $E (G)$ with the following properties.
\begin{enumerate}
 \item The vertices and edges intersect in disjoint line segments.
 \item Each such line segment lies on the boundary of precisely one vertex and precisely one edge. In particular, no two vertices intersect, and no two edges intersect.
 \item Every edge contains exactly two such line segments.
\end{enumerate}
See Figure~\ref{f1} for an example.

\begin{figure}
\centering
\subfigure[$G$]{
\labellist
 %\tiny\hair 2pt
 \pinlabel {$v$}  at  83 55
\pinlabel {$1$}  at  106 228
\pinlabel {$2$}   at    25 201
\pinlabel {$3$}  at    144 186
\pinlabel {$4$}   at   83 121
\pinlabel {$5$}   at     146 79
\pinlabel {$6$}  at   131 99
\pinlabel {$7$}   at    166 13
\pinlabel {$8$}   at  250 99
\endlabellist
\includegraphics[scale=.55]{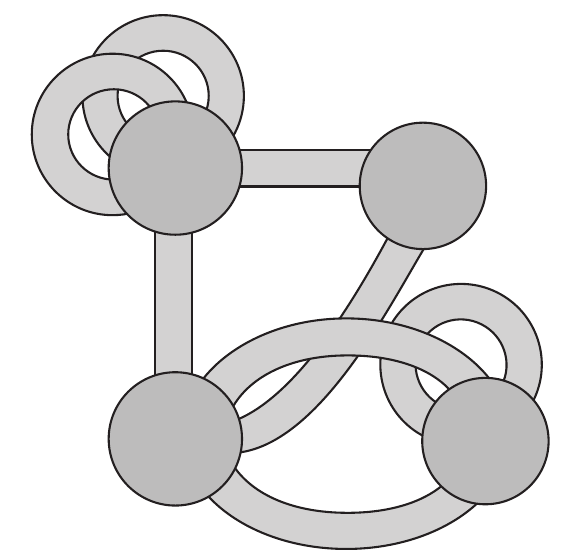}
\label{f1a}
}
\hspace{5mm}
\subfigure[$G/{\{3,8\}}  \ba \{1\} $]{
\labellist
 %\tiny\hair 2pt
\pinlabel {$2$}   at    25 201
\pinlabel {$5$}  at    142 150
\pinlabel {$4$}   at   80 121
\pinlabel {$6$}  at   173 121
\pinlabel {$7$}   at    134 29
\endlabellist
\includegraphics[scale=.55]{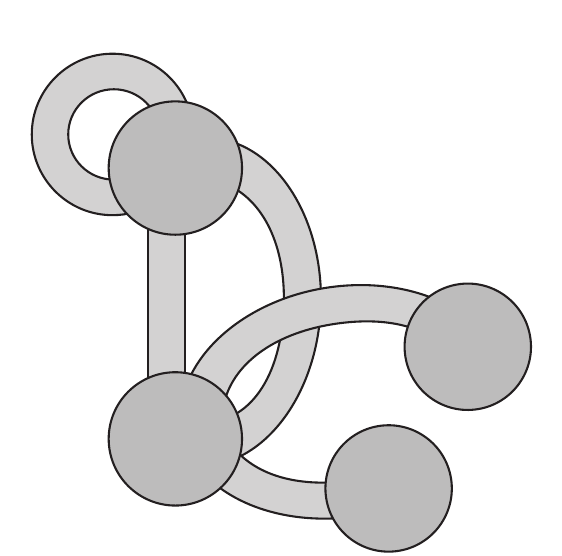}
\label{f1b}
}

%\hspace{-10mm}

\subfigure[$G^{\{1,6,7\}}=G^{\delta(\{1,6,7\})}$]{
\labellist
% \tiny\hair 2pt
\pinlabel {$1$}  at  125 211
\pinlabel {$2$}   at    53 187
\pinlabel {$3$}  at    160 171
\pinlabel {$4$}   at   109 109
\pinlabel {$5$}   at     215 109
\pinlabel {$6$}  at   161 86
\pinlabel {$8$}   at    161 51
\pinlabel {$7$}   at  161 16
\endlabellist
\includegraphics[scale=.55]{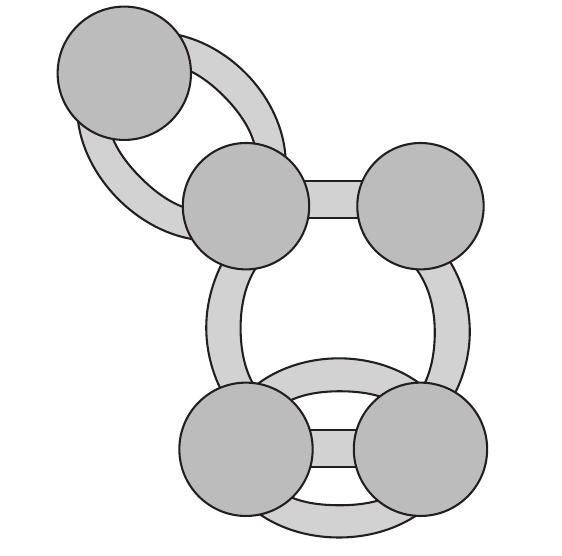}
\label{f1c}
}
\hspace{5mm}
\subfigure[$G^{\tau(\{3,8\})}$]{
\labellist
 %\tiny\hair 2pt
\pinlabel {$1$}  at  106 228
\pinlabel {$2$}   at    25 201
\pinlabel {$3$}  at    130 186
\pinlabel {$4$}   at   83 121
\pinlabel {$5$}   at     146 79
\pinlabel {$6$}  at   131 99
\pinlabel {$7$}   at    166 12
\pinlabel {$8$}   at  250 99
\endlabellist
\includegraphics[scale=.55]{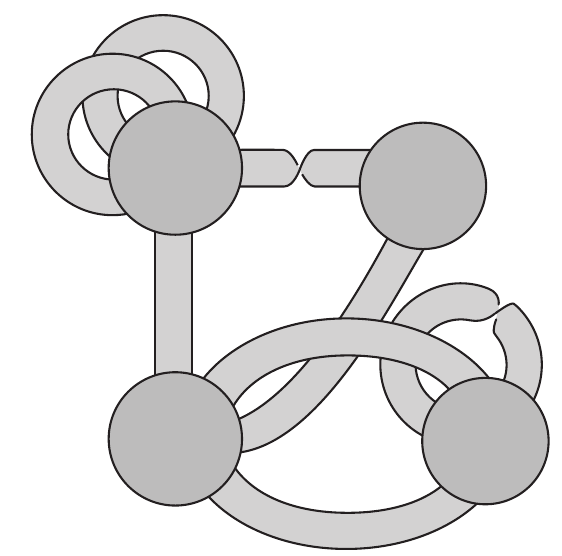}
\label{f1d}
}
\caption{An illustration of ribbon graph operations}
\label{f1}
\end{figure}
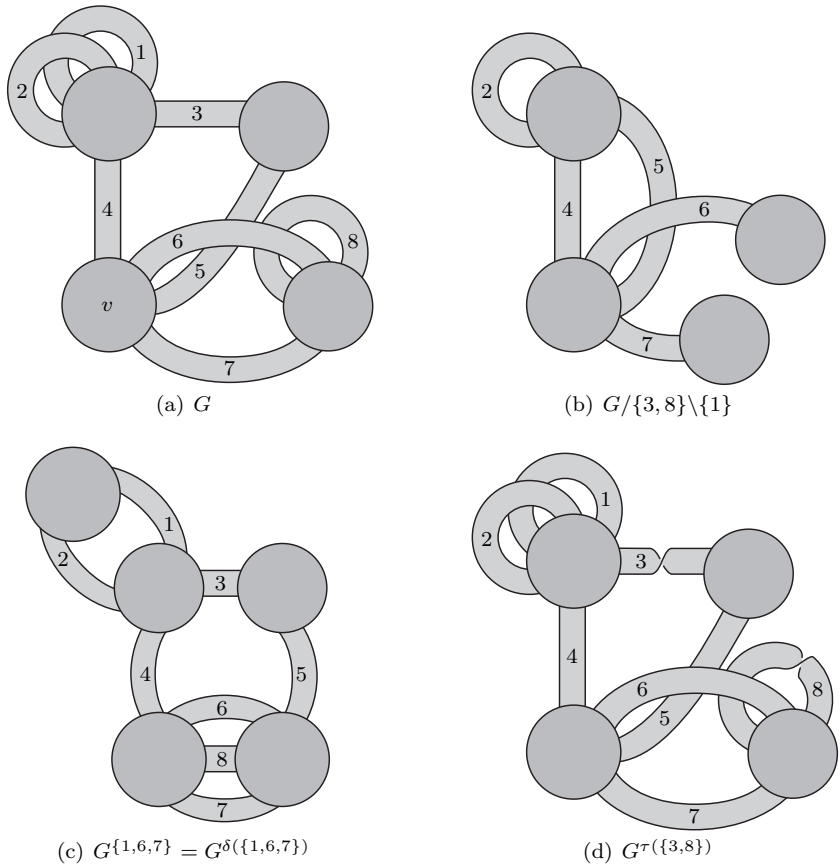

It is well-known that ribbon graphs are just descriptions of graphs cellularly embedded
in surfaces (see for example~\cite{EMMbook,GT87}). If $G$ is a cellularly embedded graph, then a ribbon
graph representation results from taking a small neighbourhood of
the cellularly embedded graph $G$, and deleting its complement. On the other hand, if $G$ is a
ribbon graph, then, topologically, it is a surface with boundary. Capping off the holes, that is, `filling in' each hole by identifying its boundary component with the boundary of a disc, results in a ribbon graph embedded in a closed surface from which a graph embedded in the surface is readily obtained. Figure~\ref{f.desc} shows an embedded graph described as both a cellularly embedded graph and a ribbon graph.

\begin{figure}[t]
\centering
\subfigure[A cellularly embedded graph $G$]{
\labellist \small\hair 2pt
\pinlabel {$1$} at   84 14
\pinlabel {$2$} at    135 19
\pinlabel {$3$} at    68 32
\pinlabel {$4$} at   115 32
\endlabellist
\raisebox{0mm}{\includegraphics[height=2.5cm]{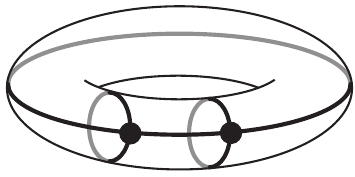}}
 \label{f.desca}
}
\qquad\qquad
\subfigure[$G$ as a ribbon graph]{
\labellist \small\hair 2pt
\pinlabel {$1$} at   60 22.7
\pinlabel {$2$} at   54 45.6
\pinlabel {$3$} at   38 7
\pinlabel {$4$} at   79 7
\endlabellist
\includegraphics[height=2.5cm]{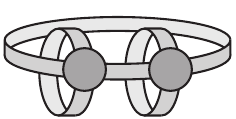}
 \label{f.descb}
}
\caption{Embedded graphs and ribbon graphs}
\label{f.desc}
\end{figure}
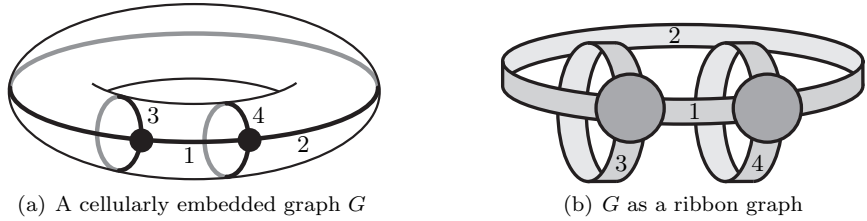

Two ribbon graphs are \emph{equivalent} if there is a homeomorphism (which should be orientation preserving if the ribbon graphs are orientable) from one to the other that preserves the vertex-edge structure, adjacency, and cyclic ordering of the half-edges at each vertex. (Note that ribbon graphs are equivalent if and only if they describe equivalent cellularly embedded graphs.)
 A ribbon graph is \emph{orientable} if it is an orientable surface, and is \emph{non-orientable} otherwise. Its \emph{genus} is its genus as a surface, and we say it is \emph{plane} if it is of genus zero (thus here we allow disconnected plane ribbon graphs).

A ribbon graph is a graph with additional structure and so standard graph terminology carries over to ribbon graphs.
If $G=(V,E)$ is a ribbon graph, then $v(G)$ and $e(G)$ denote $|V|$ and $|E|$, respectively. Furthermore, $k(G)$ denotes the number of connected components in $G$; the \emph{rank} of $G$, denoted by $r(G)$, is defined to be $v(G)-k(G)$; the \emph{nullity} of $G$, denoted by $n(G)$, is defined to be $e(G)-r(G)$; and $f(G)$ is the number of boundary components of the surface defining the ribbon graph.
The \emph{Euler genus}, $\gamma(G)$, of $G$ is the genus of $G$ if $G$ is non-orientable, and is twice its genus if $G$ is orientable. Euler's formula gives $\gamma(G)=2k(G)-v(G)+e(G)-f(G)$.

If $A\subseteq E$, then $G\backslash A$ is the \emph{ribbon subgraph} of $G=(V,E)$ obtained by \emph{deleting} the edges in $A$.
The \emph{spanning subgraph} of $G$ on $A$ is $(V,A)= G\backslash A^c$. (We will frequently use the notational shorthand $A^c:= E-A$ in the context of graphs, ribbon graphs, matroids and delta-matroids.)

For each subset $A$ of $E$, we let $k(A)$, $r(A)$, $n(A)$, $f(A)$, and $\gamma(A)$ each refer to the spanning ribbon subgraph $(V,A)$ of $G$, where $G$ is given by context.

The definition of edge contraction, introduced in~\cite{BR2,Ch1}, is a little more involved than that of edge deletion.
Let $e$ be an edge of $G$ and $u$ and $v$ be its incident vertices, which are not necessarily distinct. Then $G/e$ denotes the ribbon graph obtained as follows: consider the boundary component(s) of $e\cup u \cup v$ as curves on $G$. For each resulting curve, attach a disc, which will form a vertex of $G/e$, by identifying its boundary component with the curve. Delete the interiors of $e$, $u$ and $v$ from the resulting complex.
We say that $G/e$ is obtained from $G$ by \emph{contracting} $e$. If $A\subseteq E$, $G/A$ denotes the result of contracting all of the edges in $A$ (the order in which they are contracted does not matter).
A discussion about why this is the natural definition of contraction for ribbon graphs can be found in~\cite{EMMbook}.
The local effect of contracting an edge of a ribbon graph is shown in Table~\ref{tablecontractrg}. Note that contracting an edge in $G$ may change the number of vertices, number of components, or orientability.
For instance, if $G$ is the orientable ribbon graph with one vertex and one edge, then contracting that edge results in the ribbon graph comprising two isolated vertices.

\begin{table}[t]
\centering
\begin{tabular}{|c||c|c|c|}\hline
 &  non-loop & non-orientable loop&orientable loop\\ \hline
\raisebox{6mm}{$G$} &
\includegraphics[scale=.25]{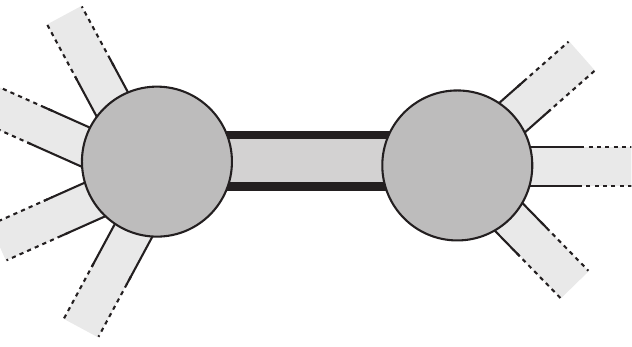} &\includegraphics[scale=.25]{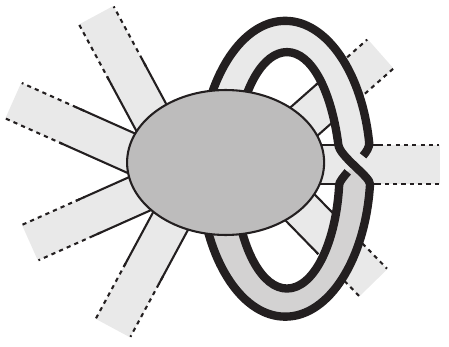} &\includegraphics[scale=.25]{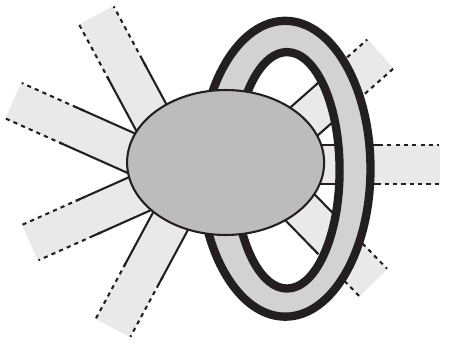}
\\ \hline
\raisebox{6mm}{$G\ba e$}
&
\includegraphics[scale=.25]{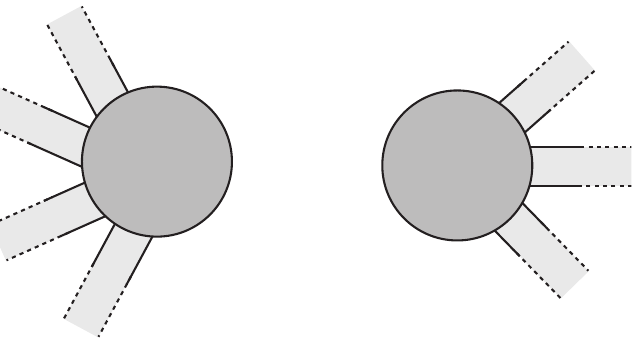} &\includegraphics[scale=.25]{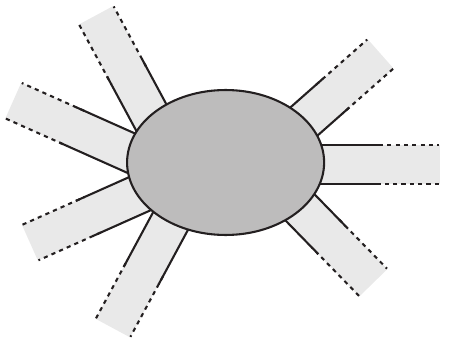}&\includegraphics[scale=.25]{ch4_35a}
\\ \hline
\raisebox{6mm}{\begin{tabular}{l} $G/e$ \end{tabular}}
&
\includegraphics[scale=.25]{ch4_35a} &\includegraphics[scale=.25]{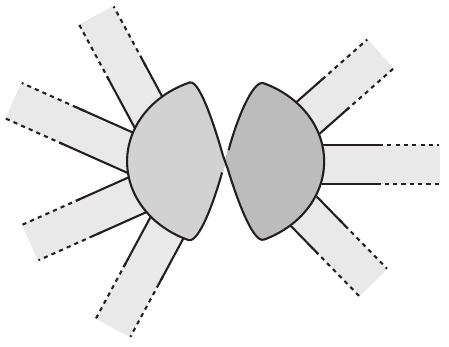}&\includegraphics[scale=.25]{ch4_38a} \\ \hline
\raisebox{6mm}{$G^{e}$} &
\includegraphics[scale=.25]{ch4_35} &\includegraphics[scale=.25]{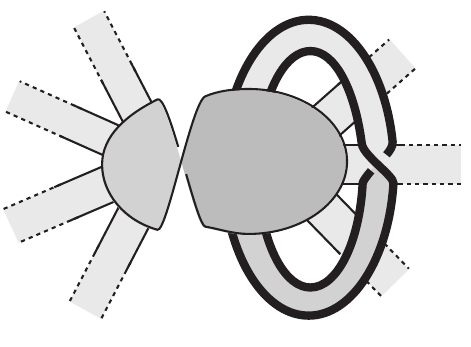} &\includegraphics[scale=.25]{ch4_38}
\\ \hline
\end{tabular}
\caption{Operations on an edge $e$ (highlighted in bold) of a ribbon graph}
\label{tablecontractrg}
\end{table}

A ribbon graph $H$ is a \emph{minor} of a ribbon graph $G$ if $H$ is obtained from $G$ by a sequence of edge deletions, edge contractions, and deletions of isolated vertices. See Figure~\ref{f1} for an example.

An edge in a ribbon graph is a \emph{bridge} if its deletion increases the number of components of the ribbon graph. It is a \emph{loop} if it is incident with only one vertex. A loop is a \emph{non-orientable loop} if, together with its incident vertex, it is homeomorphic to a M\"obius band, otherwise it is an \emph{orientable loop}.
Two cycles $C_1$ and $C_2$ in $G$ are said to be \emph{interlaced} if there is a vertex $v$  such that  $V(C_1)\cap V(C_2)=\{v\}$, and $C_1$ and $C_2$ are met in the cyclic order $C_1\,C_2\,C_1\,C_2$ when travelling round the boundary of the vertex $v$. A loop is \emph{non-trivial} if it is interlaced with some cycle in $G$. Otherwise the loop is \emph{trivial}.

Our interest here is in various notions of duality from topological graph theory. A slower exposition of the constructions here can be found in, for example,~\cite{CMNR,EMMbook}. We start with Chmutov's partial duals of~\cite{Ch1}. Let $G$ be a ribbon graph and $A\subseteq E(G)$. The \emph{partial dual} $G^{A}$ of $G$ is obtained as follows. Regard the boundary components of the spanning ribbon subgraph $(V(G),A)$ of $G$ as curves on the surface of $G$. Glue a disc to $G$ along each connected component of this curve and remove the interior of all vertices of $G$. The resulting ribbon graph is the \emph{partial dual} $G^{A}$.
The local effect of forming the partial dual with respect to an edge of a ribbon graph is shown in Table~\ref{tablecontractrg}.
 See Figure~\ref{f1} for an example.
 It is immediate from the definition (or see~\cite{Ch1}) that $G^{\emptyset}=G$ and $G/e= G^{e}\ba e$. The \emph{geometric dual} of $G$ can be defined by $G^*:=G^{E(G)}$.

Next we consider the Petrie dual (also known as the Petrial), $G^{\times}$, of $G$ (see Wilson~\cite{Wil79}).
 Let $G$ be a ribbon graph and $A\subseteq E(G)$. The \emph{partial Petrial}, $G^{\tau(A)}$, of $G$ is the ribbon graph obtained from $G$ by for each edge $e\in A$, choosing one of the two arcs $(a,b)$ where $e$ meets a vertex, detaching $e$ from the vertex along that arc giving two copies of the arc $(a,b)$, then reattaching it but by gluing $(a,b)$ to the arc $(b,a)$, where the directions are reversed (Informally, this can be thought of as adding a half-twist to the edge $e$.) The \emph{Petrie dual} of $G$ is $G^{\times}:=G^{\tau(E(G)}$. We often write $G^{\tau(e)}$ for $G^{\tau(\{e\})}$. See Figure~\ref{f1} for an example.

The partial dual and partial Petrial operations together gives rise to a group action on ribbon graphs and the concept of twisted duality from~\cite{EMM}.
Let $G$ be a ribbon graph and $A\subseteq E(G)$.
In the context of twisted duality we will use $G^{\delta(A)}$ to denote the partial dual $G^{A}$ of $G$.
Let $w=w_1w_2\cdots w_n$ be a word in the alphabet $\{\delta, \tau\}$. Then we define
 $G^{w(A)}:=(\cdots (( G^{w_1(A)} )^{w_{2}(A)} \cdots )^{w_n(A)}$.
Let $\mathfrak{G} := \langle \delta, \tau \mid \delta^2, \tau^2, (\delta\tau)^3\rangle$, which is just a presentation of the symmetric group of degree three. It was shown in~\cite{EMM} that $\mathfrak{G}$ acts on the set
$\mathcal{X} = \{ (G,A) : G\text{ a ribbon graph}, A\subseteq E(G)  \}$
by $ g(G,A) := (G^{g(A)},A)$ for $g \in \mathfrak{G} $.

Now suppose $G$ is a ribbon graph, $A,B\subseteq E(G)$, and $g,h\in \mathfrak{G}$. Define $G^{g(A)h(B)}:=\left(G^{g(A)}\right)^{h(B)}$. We say that two ribbon graphs $G$ and $H$ are {\em twisted duals} if there exist $A_1, \ldots ,A_n \subseteq E(G)$ and $g_1,\ldots, g_n \in \mathfrak{G}$ such that
$ H=G^{ g_1(A_1) g_2(A_2)\cdots g_n(A_n) }$.
Observe that,
\begin{enumerate}
\item if $A\cap B=\emptyset$, then $G^{g(A)h(B)}=G^{h(B)g(A)}$,
\item $G^{g(A)} = (G^{g(e)})^{g(A\backslash e)}$, and
\item $G^{g_1(A)}=G^{g_2(A)}$  if $g_1=g_2$ in the group $ \langle \delta, \tau \mid \delta^2, \tau^2, (\delta\tau)^3\rangle$.
\end{enumerate}
We note that Wilson's direct derivatives and opposite operators from~\cite{Wil79} result from restricting twisted duality to the whole edge set $E(G)$.

\subsection{Delta-matroids from ribbon graphs}
We briefly review the interactions between delta-matroids and ribbon graphs discussed in~\cite{CMNR} (proofs of all the results mentioned here can be found in this reference).
Let $G=(V,E)$ be a graph or ribbon graph. The \emph{cycle matroid} of $G$ is  $M(G):=(E, \mathcal{B})$ where $\mathcal{B}$ consists of the edge sets of the spanning subgraphs of $G$ that form a spanning tree when restricted to each connected component of $G$. In terms of ribbon graphs, a tree can be characterised as a genus 0 ribbon graph with exactly one boundary component. Dropping the genus 0 requirement gives a quasi-tree: a \emph{quasi-tree} is a ribbon graph with exactly one boundary component. Quasi-trees play the role of trees in ribbon graph theory, and replacing ``tree'' with ``quasi-tree'' in the definition of a cycle matroid results in a delta-matroid.
Let $G$ be a ribbon graph. Then the \emph{delta-matroid of $G$} is $D(G):=(E,\mathcal{F})$ where $\mathcal{F}$ consists of the edge sets of the spanning ribbon subgraphs of $G$ that form a quasi-tree when restricted to each connected component of $G$.

\begin{example}
Let $G$ be the ribbon graph of Figure~\ref{f1a}. Then $D(G)$ has 20 feasible sets, 10 of which are
$\{3,4,6 \}$,
$\{3,4,7\}$,
$\{3,5,6\}$,
$\{3,5,7\}$,
$\{4,5,6\}$,
$\{4,5,7\}$,
$\{3,4,5,6,7\}$,
$\{3,4,6,7,8 \}$,
$\{3,5,6,7,8\}$,
$\{4,5,6,7,8\}$. The remaining 10 are obtained by taking the union of each of these with $\{1,2\}$.
 It can be checked that $ D(G/{\{3,8\}} \ba \{1\}) = D(G)/{\{3,8\}} \ba \{1\}$ and $D(G^{\{1,6,7\}}) = D(G) \ast \{1,6,7\}$.
\end{example}

Fundamental delta-matroid operations and ribbon graph operations are compatible with each other, as in the following theorem. Part~\ref{t.compat.3} is from~\cite{CMNR}; the others are from~\cite{ab2}.
\begin{theorem}[(Bouchet~\cite{ab2}, Chun et al.~\cite{CMNR})]\label{t.compat}
Let $G=(V,E)$ be a ribbon graph. Then
\begin{enumerate}
\item \label{t.compat.1} $D(G)_{\min}=M(G)$ and  $D(G)_{\max}=(M(G^*))^*$;
\item \label{t.compat.2} $D(G)=M(G)$ if and only if $G$ is a plane ribbon graph;
\item \label{t.compat.5} $D (G)$ is a binary delta-matroid;
\item \label{t.compat.3} $D(G^A) = D(G) * A $, in particular $D(G^*)=D(G)^*$;
\item \label{t.compat.4} $D (G\ba e)= D(G)\ba e$ and $D(G/e)=D(G)/e$, for each $e\in E$.
\end{enumerate}
\end{theorem}
The significance of Theorem~\ref{t.compat}, as we will see, is that it provides the means to move between ribbon graphs and delta-matroids, giving new insights into the structure of delta-matorids.

For notational simplicity in this paper we will take advantage of the following abuse of notation. For disconnected graphs, a standard abuse of notation is to say that $T$ is a spanning tree of $G$ if the components of $T$ are spanning trees of the components of $G$. We will say that $Q$ is a \emph{spanning quasi-tree} of $G$ if the components of $Q$ are spanning quasi-trees of the components of $G$. Thus we can say that the feasible sets of $D(G)$ are the edge sets of the spanning quasi-trees of $G$. This abuse should cause no confusion.

\section{Twists of matroids}
\label{3rd}

Twists provide a way to construct one delta-matroid from another. As the class of matroids is not closed under twists, but every matroid is a delta-matroid, it provides a way to construct delta-matroids from matroids. Twisting therefore gives a way to uncover the structure of delta-matroids by translating results from the much better developed field of matroid theory.
For this reason, the class of delta-matroids that arise as twists of matroids is an important one. In this section we examine the structure of this class of delta-matroids. In particular, we provide both an excluded minor characterisation, and a rough structure theorem for this class.
Of particular interest here is the way that we are led to the results: we use ribbon graph theory to guide us. Our results provide support for the claim in this paper and in~\cite{CMNR} that ribbon graphs are to delta-matroids what graphs are to matroids.

In order to understand the class of delta-matroids that are twists of matroids, we start by looking for the ribbon graph analogue of the class of delta-matroids that are twists of matroids.
For this suppose that $G=(V,E)$ is a ribbon graph with delta-matroid $D=D(G)$. We wish to understand when $D$ is the twist of a matroid, that is, we want to determine if $D=M\ast A$ for some matroid $M=(E,\mathcal{B})$ and for some $A\subseteq E$. As twists are involutary, we can reformulate this problem as one of determining if  $D\ast B =M$ for some matroid $M$ and some $B\subseteq E$. By Theorem~\ref{t.compat}\ref{t.compat.3}, $D\ast B = D(G)\ast B =  D(G^B)$, but, by Theorem~\ref{t.compat}\ref{t.compat.2},  $D(G^B)$ is a matroid if and only if $G^B$ is a plane graph. So $D$ is a twist of a matroid if and only if $G$ is the partial dual of a plane graph.
Thus to understand the class of delta-matroids that are twists of matroids we look towards the class of ribbon graphs that are partial duals of plane graphs. Fortunately, due to connections with knot theory (see~\cite{Mo5}), this class of ribbon graphs has been fairly well studied with both a rough structure theorem and an excluded minor characterisation.

The following tells us that it makes sense to look for an excluded minor characterisation of twists of matroids.
\begin{theorem}
\label{pdminors}
The class of delta-matroids that are twists of matroids is minor-closed.
\end{theorem}
\begin{proof}
We will show that, given a matroid $M$ and a subset $A$ of $E(M)$, if $D=(E,\mathcal{F})=M*A$ and $D'$ is a minor of $D$, then $D'=M'*A'$ for some minor $M'$ of $M$ and some subset $A'$ of $E(M')$.
If $e\notin A$ then $D\setminus e = (M*A)\setminus e = (M\setminus e) *A$ and
$D/e = (M*A)/e = (M/e) *A$. On the other hand, if $e\in A$ then
$D\setminus e = (M*A)\setminus e= (M/e)*(A-e)$ and
$D/ e = (M*A)/ e= (M\setminus e)*(A-e)$.
\end{proof}

An excluded minor characterisation of partial duals of plane graphs appeared in~\cite{Moprep}.
It was shown there that a ribbon graph $G$ is a partial dual of a plane graph if and only if it contains no $G_0$-, $G_1$- or $G_2$-minor, where $G_0$ is the non-orientable ribbon graph on one vertex and one edge; $G_1$ is the orientable ribbon graph given by vertex set $\{1,2\}$, edge set $\{a,b,c\}$ with the incident edges at each vertex having the cyclic order $abc$, with respect to some orientation of $G_1$; and  $G_2$ is the orientable ribbon graph given by vertex set $\{1\}$, edge set $\{a,b,c\}$ with the cyclic order $abcabc$ at the vertex.
(The result in~\cite{Moprep} was stated for the class of ribbon graphs that present link diagrams, but this coincides with the class of partial duals of plane graphs.)

 For the delta-matroid analogue of the ribbon graph result set:
 \begin{itemize}
\item $X_0:=D(G_0)=(\{a\} ,\{ \emptyset, \{a\}\} )$;
\item $X_1:=D(G_1)=(\{a,b,c\},\{\{a\},\{b\},\{c\},\{a,b,c\}\})$;
\item $X_2:=D(G_2)=(\{a,b,c\},\{\{a,b\},\{a,c\},\{b,c\},\emptyset\})$.
\end{itemize}
Note that every twist of $X_0$ is isomorphic to $X_0$ and that every twist of $X_1$ or $X_2$ is isomorphic to either $X_1$ or $X_2$.
In particular, $X_1=X_2^*$.

Then translating the ribbon graph result into delta-matroids suggests that $X_0$, $X_1$ and $X_2$ should form the set of excluded minors for the class of delta-matroids that are twists of matroids. Previously Duchamp~\cite{adfund} had shown, but not explicitly stated, that $X_1$ and $X_2$ are the excluded minors for the class of even delta-matroids that are twists of matroids.

\begin{theorem}[(Duchamp~\cite{adfund})]
\label{expdm}
A delta-matroid $D=(E,\mathcal{F})$ is the twist of a matroid if and only if it does not have a minor isomorphic to $X_0$, $X_1$, or $X_2$.
\end{theorem}
\begin{proof}
Take a matroid $M$ and a subset $A$ of $E(M)$.
As $|B|=r(M)$ for each $B\in\mathcal{B}(M)$, we know that the sizes of the feasible sets of $M*A$ will have even parity if $r(M)$ and $A$ have the same parity, otherwise they will all have odd parity.
Thus $M*A$ is an even delta-matroid, and $X_0$ is obviously the unique excluded minor for the class of even delta-matroids.
An application of~\cite[Propositions~1.1 and 1.5]{adfund} then gives that $X_0$, $X_1$, $X_2$ is the complete list of the excluded minors of twists of matroids.
\end{proof}

We now look for a rough structure theorem for delta-matroids that are twists of matroids. Again we proceed constructively via ribbon graph theory, starting with a rough structure theorem for the class of ribbon graphs that are partial duals of plane graphs, translating it into the language of delta-matroids, then giving a proof of the delta-matroid result.

A vertex $v$ of a ribbon graph $G$ is a \emph{separating vertex} if there are non-trivial ribbon subgraphs $P$ and $Q$ of $G$ such that $(V(G),E(G))=(V(P)\cup V(Q),E(P)\cup E(Q))$, and $E(P)\cap E(Q)=\emptyset$ and $V(P)\cap V(Q)=\{v\}$. In this case we write $G=P\oplus Q$.
Let $A\subseteq E(G)$. Then we say that $A$ defines a \emph{plane-biseparation} of $G$ if all of the components of $G\ba A$ and $G\ba (E(G)-A)$ are plane and every vertex of $G$ that is incident with edges in $A$ and edges in $E(G)-A$ is a separating vertex of $G$. \begin{example}
For the ribbon graph $G$ of Figure~\ref{f1a}, $v$ is a separating vertex with $P$ and $Q$ the subgraphs induced by edges $1, \ldots, 5$ and by $6,7,8$. $G$ admits plane-biseparations. The edge sets $\{1,6,7\}$, $\{2,6,7\}$, $\{2,3,4,5,8\}$, and $\{1,3,4,5,8\}$ are exactly those that define plane-biseparations. $G^A$ is plane if and only if $A$ is one of these four sets.
\end{example}
 In~\cite{Mo5}, the following rough structure theorem was given:
\begin{theorem}[(Moffatt~\cite{Mo5})]\label{rgppd}
Let $G$ be a ribbon graph. Then the partial dual $G^A$ is a plane graph if and only if $A$ defines a plane-biseparation of $G$.
\end{theorem}

Thus we need to translate a plane-biseparation into the language of delta-matorids. Since, by Theorem~\ref{t.compat} a ribbon graph $G$ is plane if and only if $D=D(G)$ is a matroid, the requirement that $G\ba A$ and $G\ba A^c$ are plane translates to $D\ba A$ and $D\ba A^c$ being matroids. For the analogue of separability we make the following definition.
A delta-matroid is \emph{separable} if its lower matroid is disconnected.

For a ribbon graph $G$ and non-trivial ribbon subgraphs $P$ and $Q$ of $G$, we write $G=P\sqcup Q$ when $G=P\cup Q$ and $P\cap Q=\emptyset$.
If there are non-trivial ribbon subgraphs $P$ and $Q$ of $G$ and a vertex $v$ of $G$ such that $G=P\cup Q$ and $P\cap Q=\{v\}$ then we write $G=P\oplus Q$.
It was shown in~\cite{CMNR} that
if $G$ is a ribbon graph, then $D(G)$ is separable if and only if there exist ribbon graphs $G_1$ and $G_2$ such that $G=G_1 \sqcup G_2$ or $G = G_1 \oplus G_2$. Thus the condition that every vertex of $G$ that is incident with both edges in $A$ and edges in $A^c$ is a separating vertex of $G$ becomes that $A$ is separating in $D_{\min}$.
So Theorem~\ref{rgppd} may be translated to delta-matroids as follows.
\begin{theorem}
\label{charaterise_delta_matroids_twists_matroids}
Let $D$ be a delta-matroid and $A$ a non-empty proper subset of $E(D)$. Then $D*A$ is a matroid if and only if the following two conditions hold:
\begin{enumerate}
\item\label{charaterise_delta_matroids_twists_matroids.1} $A$ is separating in $D_{\min}$, and
\item\label{charaterise_delta_matroids_twists_matroids.2} $D\setminus A$ and $D\setminus A^c$ are both matroids.
\end{enumerate}
\end{theorem}
We need some preliminary results before we can prove this theorem. In the case that $D$ is the twist of a matroid, we can describe the upper and lower matroids of $D$ precisely.
\begin{lemma}\label{l.mat1}
Let $M=(E,\mathcal{B})$ be a matroid and $A$ be a subset of $E$. Let $D=M*A$. Then $D_{\min}=M/A\oplus (M\ba A^c)^*$.
\end{lemma}
\begin{proof}
Since we restrict our attention to the smallest sets of the form $B\bigtriangleup A$, where $B\in \mathcal{B}$, we need only to consider those bases of $M$ that share the largest intersection with $A$.
That is, we think of building a basis for $M$ by first finding a basis of $M\ba A^c$ and then extending that independent set to a basis of $M$.
Let $I_A$ be a basis of $M\ba A^c$ and let $B_A$ be a basis of $M$ such that $I_A\subseteq B_A$.
Then $A\bigtriangleup B_A=(B_A-I_A)\cup (A-I_A)$.
Now, $B_A-I_A$ is a basis in $M/A$ and $A-I_A$ is the complement of a basis in $M\ba A^c$, so $B_A\bigtriangleup A$ is a basis of $M/A\oplus (M\ba A^c)^*$.
Thus every basis of $D_{\min}$ can be constructed in this way. On the other hand, if $B'$ is a basis of $M/A \oplus (M\setminus A^c)^*$, then $B'=B_1 \cup B_2$ where $B_1$ is a basis of $M/A$ and $B_2$ is the complement of a basis of $M\setminus A^c$. Thus $B' \btu A= B_1 \cup (A-B_2)$ is a basis of $M$, so $B'$ is a feasible set of $D$. As all bases of $M/A \oplus (M\ba A^c)^*$ are equicardinal and each basis of $D_{\min}$ is a basis of $M/A \oplus (M\ba A^c)^*$, $|B'|=r(D_{\min})$, so $B'$ is a basis of $D_{\min}$.
\end{proof}

\begin{lemma}
\label{l.mat2}
Let $M=(E,\mathcal{B})$ be a matroid and let $A$ be a subset of $E$. Let $D=M*A$. Then $D_{\max}=M \setminus A \oplus (M/A^c)^*$.
\end{lemma}
\begin{proof}
As the feasible sets of $(D_{\max})^*$ are just the feasible sets of $(D^*)_{\min}$, we deduce that $D_{\max}=((D^*)_{\min})^*$.
Now $(D^*)_{\min}=((M*A)*E)_{\min}=((M*E)*A)_{\min}=(M^**A)_{\min}$.
Lemma~\ref{l.mat1} implies that $(M^**A)_{\min}=M^*/A\oplus (M^*\ba A^c)^*$.
Then
\[M^*/A\oplus (M^*\ba A^c)^*=(M\ba A)^*\oplus ((M/A^c)^*)^*=(M\ba A)^*\oplus (M/A^c) .\]
We deduce that $(D^*)_{\min}=(M\ba A)^*\oplus (M/A^c)$, thus $((D^*)_{\min})^*=((M\ba A)^*\oplus (M/A^c))^*$.
This last direct sum, by~\cite[Proposition 4.2.21]{Oxley11}, is equal to $M\ba A\oplus (M/A^c)^*$.
\end{proof}

The next two corollaries follow immediately from Lemma~\ref{l.mat1} and Lemma~\ref{l.mat2} and the fact that, for a given field $F$, the class of matroids representable over $F$ is closed under taking minors, duals, and direct sums.

\begin{corollary}
\label{representability}
Let $\mathcal{C}$ be a class of matroids that is closed under taking minors, duals, and direct sums. If $M\in \mathcal{C}$ and $A\subseteq E(M)$, then $(M*A)_{\min}\in \mathcal{C}$ and $(M*A)_{\max}\in \mathcal{C}$.
\end{corollary}

\begin{corollary}
\label{representability2}
Given a matroid $M$ and subset $A$ of $E(M)$, both $(M*A)_{\min}$ and $(M*A)_{\max}$ are representable over any field that $M$ is.
\end{corollary}

The proof of the following lemma can be found in~\cite{CMNR}.
\begin{lemma}\label{lem:useful}
Let $D=(E,\mathcal{F})$ be a delta-matroid, let $A$ be a subset of $E$ and let $s_0 = \min\{|B \cap A| : B\in\mathcal {B}(D_{\min})\}$. Then for any $F\in\mathcal{F}$ we have $|F\cap A|\geq s_0$.
\end{lemma}

We can now prove Theorem~\ref{charaterise_delta_matroids_twists_matroids}.
\begin{proof}[ of Theorem~\ref{charaterise_delta_matroids_twists_matroids}]
Suppose first that $M:=D*A$ is a matroid. Then $D=M*A$ and
Lemma~\ref{l.mat1} shows that
$A$ is separating in $D_{\min}$.
The feasible sets of $D\setminus A$ are given by
\[ \mathcal {F}(D\setminus A) = \{ F- A: F\in\mathcal{F}(D),\ |F\cap A|\leq |F'\cap A| \text{ for all $F'\in \mathcal{F}(D)$}\}.\]
So the feasible sets of $D\setminus A$ are obtained by deleting the elements in $A$ from those feasible sets of $D$ having smallest possible intersection with $A$. As $D=M*A$, we obtain
\[ \mathcal {F}(D\setminus A) = \{ B \cap A^c: B\in\mathcal{B}(M),\ |B\cap A|\geq |B'\cap A| \text{ for all $B'\in \mathcal{B}(M)$}\}.\]
Because all bases of $M$ have the same number of elements, we see that all feasible sets of $D\setminus A$ also have the same number of elements and consequently $D\setminus A$ is a matroid. Similarly the feasible sets of $D\setminus A^c$ form a matroid.

We now prove the converse.
Let $r= r_{D_{\min}}(A)$ and $r'= r_{D_{\min}}(A^c)$.
We will show that any feasible set $F$ of $D$ satisfies $|F\cap A| - |F\cap A^c| = r-r'$. This condition implies that all feasible sets of $D*A$ have the same size which is enough to deduce that $D*A$ is a matroid, as required.

Because $A$ is separating in $D_{\min}$, any $F_0$ in $\mathcal {F}_{\min}$
must satisfy $|F_0\cap A|=r$ and $|F_0\cap A^c|=r'$.
Now Lemma~\ref{lem:useful} implies that any feasible set $F$ of $D$ satisfies $|F\cap A| \geq r$ and $|F\cap A^c| \geq r'$. We claim that a feasible set $F$ satisfies $|F\cap A| = r$ if and only if $|F\cap A^c|=r'$. The feasible sets of $D\setminus A$ are given by
\[ \mathcal {F}(D\setminus A) = \{ F- A: F\in\mathcal{F}(D),\ |F\cap A|=r\}.\]
Following condition \ref{charaterise_delta_matroids_twists_matroids.2}, these sets form the bases of a matroid and consequently all have the same size, which must be $r'$. Therefore if $F$ is in $\mathcal F(D)$ and satisfies $|F\cap A|=r$, then $|F\cap A^c|=r'$. The converse is similar and so our claim is established.

We will now prove by induction on $k$ that if $F$ is a feasible set then $|F\cap A|=r+k$ if and only if $|F\cap A^c|=r'+k$. We have already established the base case when $k=0$. Suppose the claim is true for all $k<l$.
If $F$ is a feasible set satisfying $|F\cap A|=r+l$, then using induction, we see that $|F\cap A^c|\geq r'+l$.
Suppose then there is a feasible set $F$ satisfying $|F\cap A|=r+l$ and $|F\cap A^c|>r'+l$. Let $F_1$ be a member of $\mathcal{F}_{\min}$. So
$|F_1 \cap A|=r$ and $|F_1\cap A^c|=r'$. Now choose $F_2$ to be a feasible set with $|F_2\cap A|=r+l$, $|F_2\cap A^c|>r'+l$ and $|F_2 \cap F_1 \cap A|$ as large as possible amongst such sets. There exists $x\in (F_2- F_1)\cap A$ and clearly $x\in F_1\bigtriangleup F_2$. Hence there exists $y$ belonging to $F_1\bigtriangleup F_2$ such that $F_3 = F_2 \bigtriangleup \{x,y\}$ is feasible. However $y$ is chosen, we must have $|F_3\cap A| < |F_3\cap A^c|$. Therefore the inductive hypothesis ensures that $|F_3 \cap A| \geq |F_2 \cap A|$ and so $y\in F_1\cap A$. But then $F_3$ is a feasible set of $D$ with $|F_3\cap A|=r+l$, $|F_3\cap A'|>r'+l$ and $|F_3 \cap F_1 \cap A| > |F_2 \cap F_1 \cap A|$, contradicting the choice of $F_2$.
\end{proof}

We say that  a ribbon graph $G$ is the \emph{join} of $P$ and $Q$, written $G=P\vee Q$, if $P$ and $Q$ are two disjoint non-trivial ribbon graphs and $G$ can be obtained  by choosing an arc on the boundary of a vertex in $P$ that does not intersect an edge, and doing likewise in $Q$, then identifying the two arcs. The two vertices that were identified become a single vertex of $G$.

 It is easily deduced from Proposition~5.22 of~\cite{CMNR} that $D(G)$ is disconnected if and only if there exist ribbon graphs $G_1$ and $G_2$ such that $G=G_1 \sqcup G_2$ or $G = G_1 \vee G_2$.
In~\cite{Mo5}, it was shown that $G$ and $G^A$ are both plane graphs if and only if we can write $G=H_1\vee \cdots \vee H_l$, where each $H_i$ is plane and $A= \bigcup_{i\in I} E(H_i)$, for some $I\subseteq \{1, \ldots , l\}$.
This result extends to matroids as follows.
\begin{theorem}
\label{pdism}
Let $M=(E,\mathcal{B})$ be a matroid and $A$ be a subset of $E$.
Then $M*A$ is a matroid if and only if $A$ is separating or $A\in\{\emptyset,E\}$.
\end{theorem}
\begin{proof}
The minimum-sized sets and maximum-sized sets in $\mathcal{B}\bigtriangleup A$ have size $r(M)-r(A)+|A|-r(A)$ and $r(E-A)+|A|-(r(M)-r(E-A))$, respectively.
The collection $\mathcal{B}\bigtriangleup A$ is the set of bases of a matroid if and only if they all have equal cardinality, or equivalently,
\[
r(M)-r(A)+|A|-r(A) = r(E-A)+|A|-(r(M)-r(E-A)).
\]
Simplifying yields $r(M)=r(A)+r(E-A)$, which occurs if and only if $A$ is separating or $A \in \{\emptyset, E\}$, ~\cite[Proposition 4.2.1]{Oxley11},
\end{proof}

We complete this section by generalizing a result of Welsh~\cite{We69} regarding the connection between Eulerian and bipartite binary matroids.
A graph is said to be \emph{Eulerian} if there is a closed walk that contains each of the edges of the graph exactly once, and a graph is \emph{bipartite} if there is a partition $(A,B)$ of the vertices such that no edge of the graph has both endpoints in $A$ or both in $B$.
A \emph{circuit} in a matroid is a minimal dependent set and a \emph{cocircuit} in a matroid is a minimal dependent set in its dual.
A matroid $M$ is said to be \emph{Eulerian} if there are disjoint circuits $C_1, \ldots , C_p$ in $M$ such that $E(M)=C_1\cup \cdots \cup C_p$. A matroid is said to be \emph{bipartite} if every circuit has even cardinality.

A standard result in graph theory is that a plane graph $G$ is Eulerian if and only if $G^*$ is bipartite. This result also holds for binary matroids.
\begin{theorem}[(Welsh~\cite{We69})]
\label{dominic}
A binary matroid is Eulerian if and only if its dual is bipartite.
\end{theorem}

Once again, by considering delta-matroids as a generalization of ribbon graphs, we can determine when the twist of a binary bipartite or Eulerian matroid is either bipartite or Eulerian.
In~\cite{HM11} it was shown that, if $G$ is a plane graph with edge set $E$ and $A\subseteq E$, then
\begin{enumerate}
\item $G^A$ is bipartite if and only if the components of $G\ba A^c$
and $G^*\ba A$ are Eulerian;
\item $G^A$ is Eulerian if and only if $G\ba A^c$ and $G^*\ba A$
are bipartite.
\end{enumerate}
Recalling that, when $G$ is plane, $D(G)$ is a matroid, and that $D(G^A)=D(G)\ast A$ suggests an extension of this result to twists of binary matroids, but first we need to introduce some terminology. The circuit space $\mathcal{C}(M)$ (respectively cocircuit space $\mathcal {C^*}(M)$) of a binary matroid $M=(E,\mathcal{B})$ comprises all subsets of $E$ which can be expressed as the disjoint union of circuits (respectively cocircuits). Both spaces include the empty set. It is not difficult to see that a subset of $E$ belongs to the circuit space (respectively cocircuit) space if and only if it has even intersection with every cocircuit (respectively circuit)~\cite{Oxley11}. The bicycle space $\mathcal{BI}(M)$ of $M$ is the intersection of the circuit and cocircuit spaces.
It is not difficult to show that for any binary matroid $M$, we have $\mathcal{C}(M/A)=\{C-A : C\in \mathcal{C}(M)\}$. Furthermore $\mathcal{C}(M^*) = \mathcal{C^*}(M)$ which implies that $\mathcal{BI}(M)=\mathcal{BI}(M^*)$.

\begin{theorem}
\label{bipartite}
Let $M=(E,\mathcal{B})$ be a binary matroid, $A$ be a subset of $E$ and $D=M*A$.
\begin{enumerate}
\item If $M$ is bipartite, then $D_{\min}$ is bipartite if and only if $A\in \mathcal{BI}(M)$.
\item If $M$ is Eulerian, then $D_{\min}$ is Eulerian if and only if $A\in \mathcal{BI}(M)$.
\end{enumerate}
\end{theorem}
\begin{proof}We prove the first part.
Lemma~\ref{l.mat1} implies that $D_{\min}=M/A\oplus (M\ba A^c)^*$.
So $D_{\min}$ is bipartite exactly when both $M/A$ and $(M\ba A^c)^*$ are bipartite.
Every circuit of a matroid has even cardinality if and only if every element of its circuit space has even cardinality.
Consequently every circuit of $M/A$ has even cardinality if and only if $C-A$ has even cardinality for every $C\in \mathcal{C}(M)$. Because every circuit of $M$ has even cardinality, this occurs if and only if $C \cap A$ has even cardinality for every $C\in\mathcal{C}(M)$, which corresponds to $A\in \mathcal{C^*}(M)$.

On the other hand $(M\ba A^c)^*=M^*/(E-A)$.
This matroid is bipartite exactly when every element of $\mathcal C(M^*)$ has an even number of elements not belonging to $E-A$. Equivalently the intersection of any circuit of $\mathcal C(M^*)$ with $A$ has even cardinality which occurs if and only if $A \in \mathcal {C}^*(M^*)=\mathcal C(M)$. So $D_{\min}$ is bipartite if and only if $A\in \mathcal{BI}(M)$.

The proof of the second part is very similar and so is omitted.
\end{proof}

\begin{corollary}
Let $M=(E,\mathcal{B})$ be a binary matroid, $A$ be a subset of $E$ and $D=M*A$.
\begin{enumerate}
\item If $M$ is Eulerian, then $D_{\min}$ is bipartite if and only if $E-A\in \mathcal{BI}(M)$;
\item If $M$ is bipartite, then $D_{\min}$ is Eulerian if and only if $E-A\in \mathcal{BI}(M)$.
\end{enumerate}
\end{corollary}

\begin{proof}
The first part follows from applying Theorem~\ref{bipartite} to $M^*$ and by using Theorem~\ref{dominic}, because $M^*$ is bipartite and $D=(M^*)*(E-A)$.
The second part is similar.
\end{proof}

\section{Loop complementation and vf-safe delta-matroids}\label{petriesection}
\label{4th}
So far we have seen how the concepts of geometric and partial duality for ribbon graphs can serve as a guide for delta-matroid results on twists. In this section we continue to apply concepts of duality from topological graph theory to delta-matroid theory by  examining the delta-matroid analogue of Petrie duality and partial Petriality.

Following Brijder and Hoogeboom~\cite{BH11}, let $D=(E,\mathcal{F})$ be a set system and $e\in E$.
Then $D+e$ is defined to be the set system $(E,\mathcal{F}')$ where
\[ \mathcal{F}'= \mathcal{F} \triangle \{ F\cup e : F\in \mathcal{F} \text{ and } e\notin F  \} .\]
If $e_1, e_2 \in E$ then $(D+e_1)+e_2 = (D+e_2)+e_1 $. This means that if $A=\{e_1, \ldots , e_n\}\subseteq E$ we can unambiguously define the {\em loop complementation} of $D$ on $A$, by $D+A:= D+e_1+\cdots + e_n$.

We remark that loop complementation is particularly natural in the context of binary delta-matroids. If $D=D(C)$ for some matrix $C$ over $\mathrm{GF}(2)$, then forming $D+e$ coincides with changing the diagonal entry of $C$ corresponding to $e$ from zero to one or vice versa, and forming the delta-matroid of the resulting matrix. If $C$ is regarded as the adjacency matrix of a looped simple graph, then the loop complementation corresponds to adding or removing a loop at a vertex, explaining the name of the operation.

It is important to note that the set of delta-matroids is not closed under loop complementation. For example, let $D=(E,\mathcal{F})$ with $E=\{a,b,c\}$ and $\mathcal{F}=2^{\{a,b,c\}} \setminus \{a\}$. Then $D$ is a delta-matroid, but $D+a =(E,\mathcal{F}') $, where $\mathcal{F}'= \{\emptyset, \{a\}, \{b\}, \{c\}, \{b,c\}\}$, is not a delta-matroid, since if $F_1 = \{b,c\}$ and $F_2=\{a\}$, then there is no choice of $x$ such that $F_1 \btu \{a,x\} \in \mathcal{F}'$.
To get around this issue, we often restrict our attention to a class of delta-matroids that is closed under loop complementation. A delta-matroid $D=(E,\mathcal{F})$ is said to be {\em vf-safe} if the application of every sequence of twists and loop complementations results in a delta-matroid.
 (The operations arising from sequences of twists and loop complementations are called \emph{vertex-flips}, since in the binary case they can be viewed as operations on the vertices of a graph. Vf-safe stands for vertex-flip-safe.)
The class of
vf-safe delta-matroids is known to be minor closed and strictly contains the class of binary delta-matroids (see for example~\cite{BHpre2}). In particular, as ribbon graphic delta-matroids are binary (Theorem~\ref{t.compat}\ref{t.compat.5}), it follows that ribbon-graphic delta-matroids are also vf-safe.

The following result establishes a surprising connection, by showing that loop complementation is the delta-matroid analogue of partial Petriality.
\begin{theorem}\label{l.plus}
Let $G$ be a ribbon graph and $A\subseteq E(G)$. Then $D(G)+A = D(G^{\tau(A)})$. \end{theorem}
\begin{proof}
Without loss of generality we assume that $G$ is connected.
To prove the proposition it is enough to show that $D(G)+e = D(G^{\tau(e)})$. To do this we describe how the spanning quasi-trees of $G^{\tau(e)}$ are obtained from those of $G$.

Suppose that the boundary of the edge $e$, when viewed as a disc, consists of the four arcs $[a_1, b_1]$, $[b_1,b_2]=:\beta $, $[b_2, a_2]$, and $[a_2,a_1]=:\alpha $, where $\{\alpha, \beta\}=\{e\}\cap V(G)$ are the arcs which attach $e$ to its incident vertices (or vertex).
Let $H$ be a spanning ribbon subgraph of $G$. Consider the boundary cycle or cycles of $H$ containing the points $a_1,a_2,b_1,b_2$. If $e\in E(H)$ then the boundary components of $H\ba e$ can be obtained from those of $H$ by deleting the arcs $[a_1,b_1]$ and $[b_2,a_2]$, then adding arcs $[a_1, a_2]$ and $[b_1,b_2]$.
Similarly, the boundary components of $H^{\tau(e)}$ can be obtained from those of $H$ by deleting the arcs $[a_1,b_1]$ and $[b_2,a_2]$, then adding arcs $[a_1, b_2]$ and $[a_2,b_1]$. Each time the number of boundary components of $H\ba e$ or $H^{\tau(e)}$ is computed in the argument below this procedure is used.

To relate the spanning quasi-trees of $G$ and $G^{\tau(e)}$ let $H$ be a spanning ribbon subgraph of $G$ that contains $e$. Consider the boundary components (or component) containing the points $a_1,a_2,b_1,b_2$. If there is one component then they are met in the order $a_1,b_1, b_2,a_2 $ or $ a_1,b_2,a_2,b_1$ when travelling around the unique boundary component of $H$ starting from $a_1$ and in a suitable direction.
If there are two components, then one contains $a_1$ and $b_1$, and the other contains $a_2$ and $b_2$.
Recall that $f(Q)$ denotes the number of boundary components of a ribbon graph $Q$.
Comparing the number of boundary components of $H\ba e$ and $H^{\tau(e)}$ with those of $H$, we see that if the points $a_1,a_2,b_1,b_2$ are met in the order $a_1,b_1, b_2,a_2 $, then we have
$f(H\ba e)=f(H)+1$ and $f(H^{\tau(e)})=f(H)$.
If the points are met in the order $ a_1,b_2,a_2,b_1$ then $f(H\ba e)=f(H)$ and $f(H^{\tau(e)})=f(H)+1$.
If the points are on two boundary components then $f(H\ba e)=f(H)-1$ and $f(H^{\tau(e)})=f(H)-1$.
This means that for some integer $k$, two of $f(H)$, $f(H\ba e)$ and $f(H^{\tau(e)})$ are equal to $k$ and the other is equal to $k+1$.
Note that $f(H\ba e)=f(H^{\tau(e)}\ba e)$.

From the discussion above we can derive the following.
Suppose that $H\ba e$ is not a spanning quasi-tree of $G$. Then $H$ is a spanning quasi-tree of $G$ if and only if $H^{\tau(e)}$ is a spanning quasi-tree of $G^{\tau(e)}$.
Now suppose that $H\ba e$ is a spanning quasi-tree of $G$. Then either $H$ is a spanning quasi-tree of $G$ or $H^{\tau(e)}$ is a spanning quasi-tree of $G^{\tau(e)}$, but not both.
Finally, let $D(G)=(E,\mathcal{F})$, and recall that the feasible sets of $D(G)$ (respectively $D(G^{\tau(e)})$) are the edge sets of all of the spanning quasi-trees of $G$ (respectively $G^{\tau(e)}$). From the above, we see that
feasible sets of $D(G^{\tau(e)})$ are given by $  \mathcal{F} \triangle \{ F\cup e : F\in \mathcal{F} \text{ such that } e\notin F  \} $. It follows that $D(G)+A = D(G^{\tau(A)})$, as required.
\end{proof}

For use later, we record the following lemma. Its straightforward proof is omitted.
\begin{lemma}\label{lem:opsswitch}
Let $e$ be an element of a vf-safe delta-matroid $D=(E,\mathcal{F})$ and let $A\subseteq E-e$. Then
$(D+A)/e = (D/e)+A$ and $(D+A)\ba e = (D \ba e)+A$.
\end{lemma}

 \medskip

In~\cite{BH11} it was shown that twists, $\ast$, and loop complementation, $+$, give rise to an action of the symmetric group of degree three on set systems. If $S=(E,\mathcal{F})$ is a set system, and
$w=w_1w_2\cdots w_n$ is a word in the alphabet $\{\ast, +\}$ (note that $\ast$ and $+$ are being treated as formal symbols here), then
 \begin{equation}\label{c2.eq1b}
 (S)w :=(\cdots (( S) w_1(E) )w_{2}(E) \cdots )w_n(E).
 \end{equation}
With this, it was shown in~\cite{BH11} that the group $\mathfrak{S}=\langle * , + \mid *^2, +^2, (*+)^3 \rangle$
acts on the set of ordered pairs $\mathcal{X} = \{ (S,A) : S\text{ a set system}, A\subseteq E(G)  \}$.
Let $S=(E,\mathcal{F})$ be a set system, $A,B\subseteq E$, and $g,h\in  \mathfrak{G}$. Let $S{g(A)h(B)}:=\left(S{g(A)}\right){h(B)}$. Let $D_1=(E,\mathcal{F})$ and $D_2$ be delta-matroids. We say that  $D_2$ is a {\em twisted dual} of  $D_1$ if there  exist  $A_1, \ldots A_n \subseteq E$ and $g_1,\ldots, g_n \in  \mathfrak{S}$ such that
\[ D_2=D_1{g_1(A_1) g_2(A_2)\cdots g_n(A_n)} . \]
  It was shown in~\cite{BH11} that the following hold.
  \begin{enumerate}
  \item  If $A\cap B=\emptyset$, then $Gg(A)h(B)=Gh(B)g(A)$.
  \item $ Dg(A) = (Dg(e))g(A\backslash e)$
  \item $Dg_1(A)=Dg_2(A)$ if $g_1=g_2$ in  the group $\langle * , +  \mid *^2, +^2, (*+)^3 \rangle$.
  \end{enumerate}

\medskip

We have already shown that geometric partial duality and twists as well as Petrie duals and loop complementations are compatible. So it should come as no surprise that twisted duality for ribbon graphs and for delta-matroids are compatible as well.
\begin{theorem}\label{t.td}
Let $\mathfrak{G}=\langle \delta ,\tau \mid \delta^2, \tau^2, (\delta\tau)^3 \rangle$ and $\mathfrak{S}=\langle * , +  \mid *^2, +^2, (*+)^3 \rangle$ be two presentations of the symmetric group $\mathfrak{S}_3$; and let $\eta:  \mathfrak{G}  \rightarrow \mathfrak{S}$ be the homomorphism induced by $\eta(\delta)=*$, and $ \eta(\tau)=+$. Then if $G$ is a ribbon graph,  $A_1, \ldots A_n \subseteq E(G)$, and $g_1,\ldots, g_n \in \mathfrak{G}$ then
\[D(G^{ g_1(A_1) g_2(A_2)\cdots g_n(A_n) }) = D(G) \eta (g_1)(A_1) \eta (g_2)(A_2)\cdots \eta (g_n)(A_n) . \]
\end{theorem}
\begin{proof}
The result follows immediately from Theorem~\ref{t.compat}\ref{t.compat.3} and Theorem~\ref{l.plus}.
\end{proof}

We now return to the topic of binary and Eulerian matroids as discussed at the end of Section~\ref{3rd}.
Brijder and Hoogeboom~\cite{BH12} obtained results of a different flavour on delta-matroids obtained from Eulerian or bipartite binary matroids. To describe them we  need some additional notation.
Let $D=(E,{\mathcal{F}})$ be a set system and $A$ be a subset of $E$. Then the {\em dual pivot} on $A$, denoted by $D \bar{\ast } A$, is defined by  \[D \bar{\ast} A:= ((D\ast A)+A)\ast A.\]
From the discussion on twisted duality above, it follows that $D \bar{\ast } A = ((D +A) \ast A)+ A$. We shall use this observation and other similar consequences of twisted duality several times in this section.

Let $M=(E,\mathcal{F})$ be a binary matroid. Brijder and Hoogeboom~\cite{BH12} showed that $M$ is bipartite if and only if  $M+E$ is an even delta-matroid, and that $M$ is Eulerian if and only if  $M\bar{\ast}E$ is an even delta-matroid.
These results are interesting in the context of  ribbon graphs. In~\cite{EMM} it was shown that an orientable ribbon graph $G$ is bipartite if and only if its Petrie dual $G^{\times}$ is orientable, although, unfortunately, the result was misstated. In particular, if $G$ is plane, then  $M(G)$ is bipartite if and only if $M(G)+E(G)$ is even, which is the graphic restriction of the first part of Brijder and Hoogeboom's result.
The ribbon graph analogue of the second part is that a plane graph $G$ is Eulerian if and only if   $G^{*\times *}$ is orientable. This is indeed the case: $G$ is Eulerian if and only if $G^*$ is bipartite if and only if $(G^*)^{\times}$ is orientable if and only if  $((G^*)^{\times})^*$ is orientable. However, the result does not extend to all Eulerian ribbon graphs, for example, consider the ribbon graph consisting of one vertex and two orientable non-trivial loops.

An edge $e$ of a ribbon graph $G$ is a loop if and only if $e$ is a loop of $D(G)_{\min}$. (A loop $e$ may, however, belong to feasible sets of $D(G)$ other than those forming the bases of its lower matroid.)
Loops in ribbon graphs can be classified into several types: orientable or non-orientable, trivial or non-trivial.
The authors proved in~\cite{CMNR}
that the orientability and triviality of a loop $e$ in a ribbon graph $G$ is easily determined from $D(G)$.
For instance an edge $e$ is a trivial orientable loop of a ribbon graph $G$ if and only if $e$ is a loop in $D(G)$.
This gives a way of classifying the loops of $D(G)_{\min}$. It turns out that this classification may be usefully extended to classify the loops of the lower matroid of any delta-matroid. This is yet another example of topological graph theory guiding the theory of delta-matroids. As this classification follows exactly the classification of loops in ribbon graphs we extend the nomenclature from ribbon graphs to arbitrary delta-matroids.
The following definition is from~\cite{CMNR}.

\begin{definition}\label{def:loops}
Let $D=(E,\mathcal{F})$ be a delta-matroid.  Take $e\in E$. Then
\begin{enumerate}
\item $e$ is a {\em ribbon loop} if $e$ is a loop in $D_{\min}$;
\item a ribbon loop $e$ is  \emph{non-orientable} if $e$ is a ribbon loop in $D\ast e$ and is \emph{orientable} otherwise;
\item an orientable ribbon loop $e$ is \emph{trivial} if $e$ is in no feasible set of $D$ and is \emph{non-trivial} otherwise;
\item a non-orientable ribbon loop $e$ is \emph{trivial} if $F\btu e$ is in $\F$ for every feasible set $F\in\F$ and is \emph{non-trivial} otherwise.
\end{enumerate}
\end{definition}

Notice that a trivial orientable ribbon loop of a delta-matroid $D$ is exactly a loop of $D$.

The following result is a slight reformulation of Theorem~5.5 from~\cite{BH13} and is the key to understanding the different types of ribbon loops in a vf-safe delta-matroid. Following the notation of~\cite{BH12}, we set  $d_D:=r(D_{\min})$.
\begin{theorem}[(Brijder, and Hoogeboom~\cite{BH13})]\label{thm:minstuff}
Let $D$ be a vf-safe delta-matroid and let $e\in E$. Then
two of $D_{\min}$, $(D\ast e)_{\min}$ and $(D\bar\ast e)_{\min}$ are isomorphic. If $M_1$ is isomorphic to two matroids in $\{D_{\min},(D\ast e)_{\min},(D\bar\ast e)_{\min}\}$ and $M_2$ is isomorphic to the third then $M_1$ is formed by taking the direct sum of $M_2/e$ and the one-element matroid comprising $e$ as a loop. In particular two of $d_D$, $d_{D\ast e}$  and $d_{D\bar\ast e}$ are equal to $d$ and the third is equal to $d+1$, for some integer $d$. Finally
$e$ is a loop in precisely two of $D_{\min}$, $(D\ast e)_{\min}$ and $(D\bar\ast e)_{\min}$.
\end{theorem}

The next lemma shows that the behaviour of each type of ribbon loop in delta-matroids under the various notions of duality is exactly as predicted by the behaviour of the corresponding type of loop in ribbon graphs.

\begin{lemma}\label{l.rloopsdual}
Let $D=(E,\mathcal{F})$ be a delta-matroid and let $e\in E$. Then
\begin{enumerate}
\item \label{l.rloopsdual.1} $e$ is a coloop in $D$ if and only if $e$ is a loop in $D\ast e$,
\item \label{l.rloopsdual.2}  $e$ is neither a coloop nor a ribbon loop in $D$ if and only if $e$ is  a non-trivial orientable ribbon loop in $D\ast e$.
\end{enumerate}
If in addition $D$ is vf-safe, then
\begin{enumerate}
\addtocounter{enumi}{2}
\item \label{l.rloopsdual.3} a ribbon loop $e$ is   \emph{orientable} if and only if $e$ is a ribbon loop in $D\bar\ast e$,
\item \label{l.rloopsdual.4} a non-orientable ribbon loop is \emph{trivial} if and only if $e$ is a loop of $D+e$.
\end{enumerate}
\end{lemma}
\begin{proof}
Parts~\ref{l.rloopsdual.1} and~\ref{l.rloopsdual.4}
are straightforward and Part~\ref{l.rloopsdual.3} follows from Theorem~\ref{thm:minstuff}.
We prove Part~\ref{l.rloopsdual.2}, first proving the only if statement. As $e$ is not a ribbon loop in $D$, it is a ribbon loop in $D\ast e$ and is orientable because it is not a ribbon loop in  $(D\ast e)\ast e=D$. By \ref{l.rloopsdual.1}, it is non-trivial in $D\ast e$ since $e$ is not a coloop in $D$. Conversely, if $e$ is a non-trivial orientable ribbon loop in $D\ast e$, then as $e$ is  orientable  it is not a ribbon loop in  $(D\ast e)\ast e=D$. By \ref{l.rloopsdual.1}, it is not a coloop in $D$.
\end{proof}
A more thorough discussion of how twisting transforms the various types of ribbon loops can be found in~\cite{CMNR}.

In a ribbon graph partial Petriality changes the orientability of a loop.
The following results describe the corresponding changes in delta-matroids.
\begin{lemma}\label{l.rloopstwist}
Let $D=(E,\mathcal{F})$ be a vf-safe delta-matroid and let $e\in E$. Then $e$ is a ribbon loop in $D$ if and only if $e$ is a ribbon loop in $D+e$. Moreover a ribbon loop, $e$, is non-orientable in $D$ if and only if it is orientable in $D+e$, and $e$ is trivial in $D$ if and only if it is trivial in $D+e$.
Finally, $e$ is a coloop in $D$ if and only if $e$ is a coloop of $D+e$.
\end{lemma}
\begin{proof}The bases of $(D+e)_{\min}$ are the same as those of $D_{\min}$.
If $e$ is a ribbon loop in $D$ then no basis of $D_{\min}$ contains $e$, so $e$ is a ribbon loop in $D+e$. The converse follows because $(D+e)+e=D$.

Suppose that $e$ is a non-orientable ribbon loop in $D$. Then $e$ is a ribbon loop in $D*e$. So by the first part of this lemma, $e$ is a ribbon loop in both $D+e$ and $(D*e)+e$. By twisted duality the latter is equal to $(D+e)\bar\ast e$. Consequently, by Part~\ref{l.rloopsdual.3} of Lemma~\ref{l.rloopsdual}, $e$ is an orientable ribbon loop in $D+e$. Conversely if $e$ is an orientable ribbon loop in $D+e$, then it is a ribbon loop in $(D+e)\bar\ast e = (D*e)+e$. By the first part of this lemma, $e$ is a ribbon loop in both $D$ and $D*e$ and so it must be a non-orientable ribbon loop in $D$.

Next, the statement concerning triviality follows from the definition of trivial loops and the fact that $(D+e)+e=D$.
Finally, if $e$ is a coloop of $D$, then $e$ is in every feasible set of $D$, so $D=D+e$ and it follows that $e$ is a coloop in $D+e$. The converse follows since $(D+e)+e=D$.
\end{proof}

The next three lemmas are needed in Section~\ref{5th}. Each is stated, perhaps a little unnaturally, in terms of the dual of a vf-safe delta-matroid $D$, but this is exactly what is required later.
\begin{lemma}\label{lem:dualtwist}
Let $e$ be an element of a vf-safe delta-matroid, $D$. Then
\begin{enumerate}
\item\label{lem:dualtwist.1} $e$ is a non-trivial orientable ribbon loop in $D^*$ if and only if $e$ is a non-trivial orientable ribbon loop in $(D+e)^*$;
\item\label{lem:dualtwist.2} $e$ is a  loop in $D^*$ if and only if $e$ is a loop in $(D+e)^*$;
\item\label{lem:dualtwist.3} $e$ is a coloop in $D^*$ if and only if $e$ is a trivial non-orientable ribbon loop in $(D+e)^*$;
\item\label{lem:dualtwist.4} $e$ is neither a ribbon loop nor a coloop in $D^*$ if and only if $e$ is a non-trivial non-orientable ribbon  loop in $(D+e)^*$.
\end{enumerate}
\end{lemma}
\begin{proof}
Notice that $D^*=(D\ast(E(D)-e))\ast e$ and $(D+e)^*=((D\ast(E(D)-e))\ast e)\bar \ast e$. Consequently $(D+e)^* = (D^*)\bar\ast e$.

It follows from Part~\ref{l.rloopsdual.3} of Lemma~\ref{l.rloopsdual} that if $e$ is an orientable ribbon loop of $D^*$, then it is a ribbon loop of
$(D^*)\bar\ast e=(D+e)^*$. Moreover, as $\bar\ast$ is involutary, it follows that $e$ is an orientable ribbon loop of $D^*$
if and only if it is an orientable ribbon loop of $(D+e)^*$.
Furthermore it follows from Part~\ref{l.rloopsdual.1} of Lemma~\ref{l.rloopsdual} and the last part of Lemma~\ref{l.rloopstwist} that $e$ is loop in $D^*$ if and only if $e$ is a   loop in $(D+e)^*$. This proves the first two parts.

If $e$ is not a ribbon loop of $D^*$, then it follows from Theorem~\ref{thm:minstuff} that $e$ must be a ribbon loop of $(D^*)\bar\ast e=(D+e)^*$. Moreover, by definition, $e$ must be a non-orientable ribbon loop of $(D+e)^*$. On the other hand, if $e$ is a non-orientable ribbon loop of $(D+e)^*$ then, by definition and Theorem~\ref{thm:minstuff}, it is not a ribbon loop of $(D+e)^* \bar \ast e=D^*$.
Finally it is easily seen that Part~\ref{lem:dualtwist.1} implies that $e$ is a coloop in $D^*$ if and only if $e$ is a trivial non-orientable ribbon loop of $(D+e)^*$. This proves the last two parts.
\end{proof}

The next lemma is straightforward and its proof is omitted.
\begin{lemma}\label{lem:penroseloops}
Let $e$ be an element of a vf-safe delta-matroid, $D=(E,\mathcal{F})$ and let $A\subseteq E-e$. Then
\begin{enumerate}
\item  \label{lem:penroseloops.1}$e$ is a  loop in $D$ if and only if $e$ is a coloop in $(D+A)^*$;
\item  \label{lem:penroseloops.2}$e$ is a trivial non-orientable ribbon loop in $D$ if and only if $e$ is a coloop in $(D+A+e)^*$.
\end{enumerate}
\end{lemma}

\begin{lemma}\label{lem:penrosemins}
Let $e$ be an element of a vf-safe delta-matroid, $D$.
\begin{enumerate}
\item \label{lem:penrosemins.1} If $e$ is a non-trivial orientable ribbon loop in $D^*$ then $(D^*\ba e)_{\min} = ((D+e)^*\ba e)_{\min}$.
\item \label{lem:penrosemins.2}  If $e$ is neither a ribbon loop nor a coloop in $D^*$ then $(D^*/ e)_{\min} = ((D+e)^*\ba e)_{\min}$.
\item  \label{lem:penrosemins.3} If $e$ is a non-trivial non-orientable ribbon loop in $D^*$ then $(D^*\ba e)_{\min} = ((D+e)^*/ e)_{\min}$.
\end{enumerate}
\end{lemma}

\begin{proof} We show first that \ref{lem:penrosemins.1}  holds.
If $e$ is a ribbon loop in $D^*$ then $(D^*\ba e)_{\min}=(D^*)_{\min}$. If $e$ is a non-trivial orientable loop in $D^*$, it follows from Theorem~\ref{thm:minstuff} and Part~\ref{l.rloopsdual.3} of Lemma~\ref{l.rloopsdual} that $(D^*)_{\min} = (D^* \bar \ast e)_{\min}$. Using twisted duality
\[(D^* \bar \ast e)_{\min} = (((((D*(E(D)-e))*e)*e)+e)*e)_{\min} = ((D+e)^*)_{\min}.\]
By Part~\ref{lem:dualtwist.1} of Lemma~\ref{lem:dualtwist}, $e$ is a non-trivial orientable ribbon loop in $(D+e)^*$ and so $((D+e)^*)_{\min}=((D+e)^*\ba e)_{\min}$.

Next, we show that \ref{lem:penrosemins.2}  holds.
If $e$ is neither a ribbon loop nor a coloop in $D^*$, then by Part~\ref{lem:dualtwist.4} of Lemma~\ref{lem:dualtwist}, $e$ is a non-trivial non-orientable ribbon loop in $(D+e)^*$. Consequently $((D+e)^*\ba e)_{\min} = ((D+e)^*)_{\min}$. Because $e$ is non-orientable,
it follows from Theorem~\ref{thm:minstuff} and Part~\ref{l.rloopsdual.3} of Lemma~\ref{l.rloopsdual} that
$((D+e)^*)_{\min} = (((D+e)^*)*e)_{\min}$ and this can be shown to be equal to $((D^*\ast e)+e)_{\min}$. Because $e$ is not a ribbon loop in $D^*$, the lower matroids of $D^*/e$ and $(D^*\ast e)+e$ coincide.

We conclude our proof by showing that \ref{lem:penrosemins.3} holds.
By Part~\ref{lem:dualtwist.4} of Lemma~\ref{lem:dualtwist} with $D$ and $D+e$ interchanged, $e$ is neither a ribbon loop nor a coloop of $(D+e)^*$.
Applying \ref{lem:penrosemins.2}  with $D$ replaced by $D+e$, we get $((D+e)^*/ e)_{\min} = (((D+e)+e)^*\ba e)_{\min}$ or equivalently $((D+e)^*/ e)_{\min} = (D^*\ba e)_{\min}$ as required.
\end{proof}

\section{The Penrose and characteristic  polynomials}
\label{5th}

The Penrose polynomial was defined implicitly by Penrose in~\cite{Pen71}  for plane graphs, and was extended to all ribbon graphs (equivalently, all embedded graphs)  in~\cite{EMM11a}. The advantage of considering the Penrose polynomial of ribbon graphs, rather than just plane graphs, is that it reveals new properties of the Penrose polynomial (of both plane and non-plane graphs)  that cannot be realised exclusively in terms of plane graphs.
The (plane) Penrose polynomial has been defined in terms of bicycle spaces, left-right facial walks, or states of a medial graph. Here, as in~\cite{EMM11b}, we define it in terms of partial Petrials.

Let  $G$ be a ribbon graph. Then the {\em Penrose polynomial}, $P(G;\lambda)\in \mathbb{Z}[\lambda]$, is defined by
\[P(G;\lambda) := \sum_{A\subseteq E(G)}   (-1)^{|A|} \lambda^{f(G^{\tau(A)})}.\]
Recall, from Section~\ref{dbh}, that $f(G)$ denotes the number of boundary components of  $G$.

The Penrose polynomial has been extended to both matroids and delta-matroids.  In~\cite{AM00} Aigner and Mielke defined the Penrose polynomial of a binary matroid $M=(E,{\mathcal{F}})$ as
\begin{equation}\label{pdefmq} P(M;\lambda)  = \sum_{X\subseteq E}  (-1)^{|X|} \lambda^{\dim(B_M(X))},  \end{equation}
where $B_M(X)$ is the binary vector space formed of the incidence vectors of the sets in the collection
\[ \{ A \in \mathcal{C}(M) : A\cap X \in \mathcal{C}^*(M)\}.\]
Brijder and Hoogeboom defined the Penrose polynomial in greater generality for vf-safe delta-matroids in~\cite{BH12}.
Recall  that if $D=(E,{\mathcal{F}})$ is a vf-safe delta-matroid, and  $X\subseteq E$, then the {\em dual pivot} on $X$ is  $D \bar{\ast } X:=  ((D\ast X)+X)\ast X$, and that $d_D:=r(D_{\min})$.
The {\em Penrose polynomial} of  a vf-safe delta-matroid $D$ is then
\begin{equation}\label{pdefde} P(D;\lambda)  := \sum_{X\subseteq E}  (-1)^{|X|} \lambda^{d_{D\ast E \bar{\ast} X}}.  \end{equation}
It was shown in~\cite{BH12} that when the delta-matroid $D$ is a binary matroid, Equations~\eqref{pdefmq} and~\eqref{pdefde} agree.
Furthermore, our next result shows that  Penrose polynomials of matroids and delta-matroids are compatible with their ribbon graph counterparts.
\begin{theorem}\label{t.pencom}
Let $G$ be a  ribbon graph and $D(G)$ be its ribbon-graphic delta-matroid. Then
 \[P(G;\lambda) =  \lambda^{k(G)} P(D(G);\lambda).\]
\end{theorem}
\begin{proof}
Let $A$ be a subset of $E(G)$.
We have
\begin{align*}
D(G)\ast E\bar{\ast}A &= D(G)\ast E\ast A +A\ast A
 = D(G^{\delta(E) \delta(A)    \tau (A) \delta(A)})
\\&= D(G^{\delta(A^c) \delta(A)  \delta(A)   \tau (A) \delta(A)})
= D(G^{\tau (A) \delta(E)}) .
\end{align*}
Then as $D(G^{\tau (A) \delta(E)})_{\min} = M(G^{\tau (A) \delta(E)})$, we have
\[  r(D(G^{\tau (A) \delta(E)})_{\min} ) =  v(G^{\tau (A) \delta(E)}) -k(G^{\tau (A) \delta(E)})=   f(G^{\tau (A)})  -k(G), \]
using that the number of vertices of a ribbon graph is equal to the number of boundary components of its dual.
The equality of the two polynomials follows.
\end{proof}

A very desirable property of a graph polynomial  is that it satisfies a recursion relation  that reduces a graph to a linear combination of ``elementary" graphs, such as isolated vertices.  The well-known deletion-contraction reduction meets this requirement in the case of the Tutte polynomial. In~\cite{EMM11a} it was shown that the Penrose polynomial of a ribbon graph admits such a relation. If $G$ is a ribbon graph, and $e\in E(G)$, then
\begin{equation}\label{e.pdc}
P(G; \lambda)=  P(G/e; \lambda) -P (G^{\tau(e)}/e; \lambda).
\end{equation}
Recall that for a ribbon graph $G$ and non-trivial ribbon subgraphs $P$ and $Q$ of $G$, we write $G=P\sqcup Q$ when $G$ is the disjoint union of $P$ and $Q$, that is,  when $G=P\cup Q$ and $P\cap Q=\emptyset$.
The preceding identity together with the multiplicativity of the Penrose polynomial, \[P(G_1\sqcup G_2)= P(G_1) \cdot P(G_2),\]   and its value $\lambda$ on an isolated vertex provides a recursive definition of the Penrose polynomial.
The Penrose polynomial of a  vf-safe delta-matroid also admits a recursive definition.
\begin{proposition}[(Brijder and Hoogeboom~\cite{BH12})]
\label{p.pdc}
Let $D=(E,{\mathcal{F}})$ be a vf-safe delta-matroid and $e\in E$.
\begin{enumerate}
 \item \label{p.pdc.1}If $e$ is a  loop, then
$P(D;\lambda) = (\lambda-1)P(D/e;\lambda)$.
\item \label{p.pdc.2}If $e$ is a trivial non-orientable ribbon loop, then
$P(D;\lambda) = -(\lambda-1)P((D+e)/e;\lambda)$.
\item \label{p.pdc.3}If $e$ is not a trivial ribbon loop, then
$P(D;\lambda)= P(D/e;\lambda) - P((D+e)/e;\lambda)$.
\item \label{p.pdc.4}If $E=\emptyset$, then $P(D;\lambda)=1$.
\end{enumerate}
\end{proposition}
 The recursion relation above for $P(D)$ replaces  $(D\bar\ast e)\ba e$ for $(D+e)/e$ in its statement in~\cite{BH12}, but it is easy to see that $(D+e)/e$ and $(D\bar\ast e)\ba e$ have exactly the same feasible sets. We have used $(D+e)/e$ rather than $(D\bar\ast e)\ba e$ to highlight the compatability with Equation~\eqref{e.pdc}.

Observe that Equation~\eqref{e.pdc} and a recursive definition for the ribbon graph version of the Penrose polynomial can be recovered as a special case of  Proposition~\ref{p.pdc} via Theorem~\ref{t.pencom}.
It is worth noting that Equation~\eqref{e.pdc} cannot be restricted  to the class of plane graphs, and analogously that   Proposition~\ref{p.pdc} cannot be restricted  to binary matroids. Thus   restricting the polynomial to either of these classes, as was historically done, limits the possibility of inductive proofs of many results. This further illustrates the advantages of the more general settings of ribbon graphs or delta-matroids.

Next, we show that the Penrose polynomial of a delta-matroid can be expressed in terms of the characteristic polynomials of associated matroids.
The \emph{characteristic polynomial}, $\chi(M;\lambda)$, of a matroid $M=(E,\mathcal{B})$ is defined by
\[ \chi(M;\lambda) := \sum_{A\subseteq E}  (-1)^{|A|} \lambda^{r(M)-r(A)}.  \]
The characteristic polynomial is known to satisfy deletion-contraction relations (see, for example,~\cite{We76}).
\begin{lemma}\label{lem:char}
Let $e$ be an element of a matroid $M$.
\begin{enumerate}
\item \label{lem:char.1}If $e$ is a loop, then $\chi(M;\lambda)=0$.
\item \label{lem:char.2}If $e$ is a coloop, then $\chi(M;\lambda) = (\lambda-1)\chi(M/e;\lambda) = (\lambda-1)\chi(M\setminus e;\lambda)$.
\item \label{lem:char.3}If $e$ is neither a loop nor a coloop, then $\chi(M;\lambda) = \chi(M\setminus e;\lambda) - \chi(M/e;\lambda)$.
\end{enumerate}
\end{lemma}
We define the \emph{characteristic polynomial}, $\chi(D;\lambda)$, of a delta-matroid $D$ to be $\chi(D_{\min};\lambda)$. When $D$ is a matroid, $D=D_{\min}$, so this definition is consistent with the definition above of the characteristic polynomial of a matroid.

To keep the notation manageable, we define $D^{\pi(A)}$ to be $(D+A)^*$.
\begin{theorem}\label{thm.pchi}
Let $D=(E,\mathcal{F})$ be a vf-safe delta-matroid. Then
\begin{equation}\label{eqn:penrose} P(D;\lambda) = \sum_{A\subseteq E} (-1)^{|A|}  \chi(D^{\pi(A)}  ;\lambda).   \end{equation}
\end{theorem}
\begin{proof}Throughout the proof, we make frequent use of the fact that if $e$ is an element of a delta-matroid $D$, then $(D*e)/e=D\setminus e$ and $(D*e)\setminus e=D/ e$.
The proof proceeds by induction on the number of elements of $E$. If $E=\emptyset$ then both sides of Equation~\eqref{eqn:penrose} are equal to 1. So assume that $E\ne \emptyset$ and let $e\in E$.

We have
\begin{equation} \label{eq:pen} \sum_{A\subseteq E} (-1)^{|A|} \chi(D^{\pi(A)};\lambda) = \sum_{A\subseteq E-e} (-1)^{|A|} \chi(D^{\pi(A)};\lambda) - \sum_{A\subseteq E-e} (-1)^{|A|} \chi(D^{\pi(A \cup e)};\lambda).\end{equation}
Suppose that $e$ is a loop of $D$.
Then
by Lemmas~\ref{lem:dualtwist}\ref{lem:dualtwist.3} and~\ref{lem:penroseloops}\ref{lem:penroseloops.1}, $e$ is a coloop in $D^{\pi(A)}$ and a trivial non-orientable ribbon loop in $D^{\pi(A \cup e)}$.
So for each $A\subseteq E-e$,
$\chi(D^{\pi(A \cup e)};\lambda)=0$ by Lemma~\ref{lem:char}\ref{lem:char.1}, and by Lemma~\ref{lem:char}\ref{lem:char.2}
\[
\chi(D^{\pi(A)};\lambda) =  (\lambda-1) \chi(((D^{\pi(A)})_{\min})/e;\lambda).
\]
Since $e$ is a coloop in $D^{\pi(A)}$,
we have $(D^{\pi(A)})_{\min}/e = (D^{\pi(A)}/e)_{\min} = ((D+A)^*/e)_{\min}$.
Using Lemma~\ref{lem:opsswitch},
\[((D+A)^*/e)_{\min}= ((D\ba e +A)^*)_{\min} = ((D/e +A)^*)_{\min} = ((D/e)^{\pi(A)})_{\min},\]
where the penultimate equality holds because $e$ is a loop in $D$ and hence $D\ba e= D/e$.
Therefore $\chi(((D^{\pi(A)})_{\min})/e;\lambda) = \chi((D/e)^{\pi(A)};\lambda)$.
Hence \[\sum_{A\subseteq E} (-1)^{|A|} \chi(D^{\pi(A)};\lambda) = (\lambda-1)\sum_{A\subseteq E-e} (-1)^{|A|}\chi((D/e)^{\pi(A)};\lambda).\] Using induction and Proposition~\ref{p.pdc}\ref{p.pdc.1}, this equals $(\lambda-1)P(D/e;\lambda) = P(D;\lambda)$.

Next suppose that $e$ is a trivial non-orientable ribbon loop of $D$. Then
by Lemmas~\ref{lem:dualtwist}\ref{lem:dualtwist.3} and~\ref{lem:penroseloops}\ref{lem:penroseloops.2}, $e$ is a trivial non-orientable ribbon loop in $D^{\pi(A)}$ and a coloop in $D^{\pi(A \cup e)}$. So for each $A\subseteq E-e$, using Lemma~\ref{lem:char}\ref{lem:char.2} and Lemma~\ref{lem:opsswitch},
\begin{align*}    \chi(D^{\pi(A \cup e)};\lambda)  &= (\lambda-1) \chi(((D^{\pi(A \cup e)})_{\min})/e;\lambda)= (\lambda-1)\chi(((D+e)\setminus e)^{\pi(A)};\lambda)\\ &= (\lambda-1)\chi(((D+e)/e)^{\pi(A)};\lambda), \end{align*}
where the last equality follows from the fact that $e$ is a loop of $D+e$.
Furthermore, $\chi(D^{\pi(A)};\lambda)=0$,  by Lemma~\ref{lem:char}\ref{lem:char.1}.
Hence \[\sum_{A\subseteq E} (-1)^{|A|} \chi(D^{\pi(A)};\lambda) = -(\lambda-1)\sum_{A\subseteq E-e} (-1)^{|A|} \chi(((D+e)/e)^{\pi(A)};\lambda).\] Using induction and Proposition~\ref{p.pdc}\ref{p.pdc.2},
this equals $-(\lambda-1)P((D+e)/e;\lambda) = P(D;\lambda)$.

We have covered the cases where $e$ is a trivial ribbon loop in $D$. So now we assume that this is not the case. Using induction,  Proposition~\ref{p.pdc}\ref{p.pdc.3}, and Lemma~\ref{lem:opsswitch} we have
\begin{align*}
P(D;\lambda) &= P(D/e;\lambda) - P((D+e)/e;\lambda) \\ &=\sum_{A\subseteq E-e} (-1)^{|A|} \chi((D/e)^{\pi(A)};\lambda) - \sum_{A\subseteq E-e} (-1)^{|A|} \chi(((D+e)/e)^{\pi(A)};\lambda).\end{align*}
We will show that for each $A\subseteq E-e$
\begin{equation} \label{eqn:penrosekey} \chi((D/e)^{\pi(A)};\lambda) - \chi(((D+e)/e)^{\pi(A)};\lambda) = \chi(D^{\pi(A)};\lambda) - \chi(D^{\pi(A\cup e)};\lambda),\end{equation}
which, combined with Equation~\eqref{eq:pen}, will be enough to complete the proof of the theorem. There are four cases depending on the role of $e$ in $D^{\pi(A)}$.

First, suppose that $e$ is a  loop in $D^{\pi(A)}$.
Now by Lemma~\ref{lem:opsswitch}, $(D/e)+A = (D+A)/e$ and $((D+e)/e)+A = (D+A+e)/e$. Moreover, $e$ is a coloop in $D+A$ and so $D+A+e=D+A$. Hence
\[\chi((D/e)^{\pi(A)};\lambda) = \chi(((D+e)/e)^{\pi(A)};\lambda) \qquad \text{and} \qquad \chi(D^{\pi(A)};\lambda) = \chi(D^{\pi(A\cup e)};\lambda).\] Therefore Equation~\eqref{eqn:penrosekey} holds.

Second, suppose that $e$ is a non-trivial orientable ribbon loop in $D^{\pi(A)}$. Then by Lemma~\ref{lem:dualtwist}\ref{lem:dualtwist.1}, $e$ is also a non-trivial orientable ribbon loop in $D^{\pi(A\cup e)}$. So $e$ is a ribbon loop of both $D^{\pi(A)}$ and $D^{\pi(A\cup e)}$. Consequently $\chi(D^{\pi(A)};\lambda) = \chi(D^{\pi(A\cup e)};\lambda)=0$. On the other hand, by Lemma~\ref{lem:opsswitch}, $(D/e)^{\pi(A)} = (D^{\pi(A)})\ba e$ and $((D+e)/e)^{\pi(A)} = (D^{\pi(A\cup e)})\ba e$. Applying Lemma~\ref{lem:penrosemins}\ref{lem:penrosemins.1} with $D$ replaced by $D+A$ shows that
$\chi(D^{\pi(A)}\setminus e;\lambda) = \chi((D+e)^{\pi(A)}\setminus e;\lambda)$ and hence
$\chi((D/e)^{\pi(A)};\lambda) = \chi(((D+e)/e)^{\pi(A)};\lambda)$. Therefore Equation~\eqref{eqn:penrosekey} holds.

Third, suppose that $e$ is neither a coloop nor a ribbon loop in $D^{\pi(A)}$. Then by Lemma~\ref{lem:dualtwist}\ref{lem:dualtwist.4}, $e$ is a non-trivial non-orientable ribbon loop in $D^{\pi(A\cup e)}$. Therefore $\chi(D^{\pi(A\cup e)};\lambda)=0$.
By Lemma~\ref{lem:char}\ref{lem:char.3}, we have
\[\chi(D^{\pi(A)};\lambda) = \chi(((D^{\pi(A)})_{\min})\ba e;\lambda) - \chi(((D^{\pi(A)})_{\min})/ e;\lambda).\]
We have $((D^{\pi(A)})_{\min})\ba e = ((D^{\pi(A)})\ba e)_{\min}$. Using Lemma~\ref{lem:opsswitch} this is in turn equal to $((D/e)^{\pi(A)})_{\min}$.
Consequently $\chi(((D^{\pi(A)})_{\min})\ba e;\lambda) = \chi((D/e)^{\pi(A)};\lambda)$. Because $e$ is not a ribbon loop in $D^{\pi(A)}$, $((D^{\pi(A)})_{\min})/e = ((D^{\pi(A)})/e)_{\min}$. By Lemma~\ref{lem:penrosemins}\ref{lem:penrosemins.2} this equals $((D^{\pi(A)}+e)\ba e)_{\min}$ which by Lemma~\ref{lem:opsswitch} equals $(((D+e)/e)^{\pi(A)})_{\min}$. Hence $\chi(((D^{\pi(A)})_{\min})/e;\lambda) = \chi(((D+e)/e)^{\pi(A)};\lambda)$ and Equation~\eqref{eqn:penrosekey} follows.

The final case is similar. Suppose that $e$ is a non-trivial non-orientable ribbon loop in $D^{\pi(A)}$. Then by Lemma~\ref{lem:dualtwist}\ref{lem:dualtwist.4} applied to $D+A+e$, $e$ is neither a coloop nor a ribbon loop in $D^{\pi(A\cup e)}$. We have $\chi(D^{\pi(A)};\lambda)=0$.
By Lemma~\ref{lem:char}\ref{lem:char.3}, we have
\[\chi(D^{\pi(A\cup e)};\lambda) = \chi(((D^{\pi(A\cup e)})_{\min})\ba e;\lambda) - \chi(((D^{\pi(A\cup e)})_{\min})/ e;\lambda).\]
We have $((D^{\pi(A\cup e)})_{\min})\ba e = ((D^{\pi(A\cup e)})\ba e)_{\min}$. Using Lemma~\ref{lem:opsswitch} this is in turn equal to $(((D+e)/e)^{\pi(A)})_{\min}$.
Consequently $\chi(((D^{\pi(A\cup e)})_{\min})\ba e;\lambda) = \chi(((D+e)/e)^{\pi(A)};\lambda)$. Because $e$ is not a ribbon loop in $D^{\pi(A\cup e)}$, $((D^{\pi(A\cup e)})_{\min})/e = ((D^{\pi(A\cup e)})/e)_{\min}$. By Lemma~\ref{lem:penrosemins}\ref{lem:penrosemins.3} this equals $(D^{\pi(A)}\ba e)_{\min}$. Using Lemma~\ref{lem:opsswitch}, this equals $((D/e)^{\pi(A)})_{\min}$. Hence $\chi(((D^{\pi(A\cup e)})_{\min})/e;\lambda) = \chi((D/e)^{\pi(A)};\lambda)$ and Equation~\eqref{eqn:penrosekey} follows.

Therefore the result follows by induction.
\end{proof}

If $G$ is a graph or ribbon graph and $M(G)$ its cycle matroid, then $\chi(M(G);\lambda)=\lambda^{-k(G)} \chi(G;\lambda)$, where the $\chi$ on the right-hand side refers to the chromatic polynomial of $G$. Combining this fact with Theorem~\ref{t.pencom} allows us to recover the following result as a special case of Theorem~\ref{thm.pchi}.
\begin{corollary}[(Ellis-Monaghan and Moffatt~\cite{EMM11a,EMM11b})]
Let $G$ be a ribbon graph. Then
\[  P(G;\lambda) = \sum_{A\subseteq E(G)}  (-1)^{ |A|}  \chi ((   G^{\tau(A)}   )^*   ;\lambda)  ,\]
where  $\chi (H;\lambda)$ denotes the chromatic polynomial of $H$.
\end{corollary}
In fact, in keeping with the spirit of this paper, it was the existence of this ribbon graph polynomial identity that led us to formulate Theorem~\ref{thm.pchi}.

\medskip

We end this paper with some results regarding the transition polynomial.
As observed by Jaeger in~\cite{Ja90}, the Penrose polynomial of a plane graph arises as a specialization of the transition polynomial, $q(G; W,t)$.  Ellis-Monaghan and Moffatt introduced a version of the transition polynomial for ribbon graphs (or, equivalently,  embedded graphs), called the topological transition polynomial, in~\cite{EMM}.  This polynomial provides a general framework for understanding the Penrose polynomial $P(G)$ and the ribbon graph polynomial $R(G)$ as well as some knot and virtual knot polynomials.

Let $E$ be a set. We define $\mathcal{P}_3(E) :=  \{  (E_1,E_2,E_3)  :   E=E_1\cup E_2 \cup E_3, E_i \cap E_j=\emptyset \text{ for each }  i\neq j   \}$. That is, $\mathcal{P}_3(E)$ is the set of ordered partitions of $E$ into three, possibly empty, blocks.

Let $G$ be a ribbon graph and $R$ be a commutative ring with unity. A {\em weight system} for $G$, denoted $(\boldsymbol\alpha, \boldsymbol\beta, \boldsymbol\gamma)$, is a set of ordered triples of elements in $R$ indexed by $E=E(G)$. That is,
$(\boldsymbol\alpha, \boldsymbol\beta, \boldsymbol\gamma)  :=
 \{ (\alpha_e, \beta_e, \gamma_e) :  e\in E, \text{ and }\alpha_e, \beta_e, \gamma_e\in R\}$.
   The \emph{topological transition polynomial}, $Q(G, (\boldsymbol\alpha, \boldsymbol\beta, \boldsymbol\gamma) , t) \in R[t] $, is defined by
   \begin{equation}\label{topotransdef}
Q(G, (\boldsymbol\alpha, \boldsymbol\beta, \boldsymbol\gamma) , t) :=\sum_{(A,B,C) \in \mathcal{P}_3( E(G))}    \Big( \prod_{e\in A}\alpha_e\Big) \Big(  \prod_{e\in B}  \beta_e\Big) \Big( \prod_{e\in C}   \gamma_e\Big)  t^{f( G^{\tau(C)}\setminus B)}.
\end{equation}

If, in the set of ordered triples $(\boldsymbol\alpha, \boldsymbol\beta, \boldsymbol\gamma)$, we have $(\alpha_e, \beta_e, \gamma_e)=(\alpha, \beta, \gamma)$ for all $e\in E(G)$, then we write $(\alpha, \beta, \gamma)$ in place of $(\boldsymbol\alpha, \boldsymbol\beta, \boldsymbol\gamma)$.

On the delta-matroid side, in~\cite{BHpre2} the transition polynomial of a vf-safe delta-matroid was introduced. Suppose $D=(E,{\mathcal{F}})$ is a delta-matroid. A weight system $(\boldsymbol\alpha, \boldsymbol\beta, \boldsymbol\gamma)$ for $D$ is defined just as it was for ribbon graphs above.
Then the {\em transition polynomial} of a vf-safe delta-matroid is defined by
   \begin{equation}\label{deltatransdef}  Q_{(\boldsymbol\alpha, \boldsymbol\beta, \boldsymbol\gamma)}(D;t)  =   \sum_{(A,B,C) \in \mathcal{P}_3( E)}    \Big( \prod_{e\in A}\alpha_e\Big) \Big(  \prod_{e\in B}  \beta_e\Big) \Big( \prod_{e\in C}   \gamma_e\Big) t^{d_{D\ast B \bar{\ast} C}}.   \end{equation}
Again we use the notation $(\alpha, \beta, \gamma)$ to denote the weight system in which each  $e\in E$ has weight $(\alpha, \beta, \gamma)$.

We will need a twisted duality relation for the transition polynomial. In order to state this relation, we  introduce a little notation.
Let  $(\boldsymbol\alpha, \boldsymbol\beta, \boldsymbol\gamma) =  \{(\alpha_e, \beta_e, \gamma_e): {e\in E}\}  $ be a weight system for a delta-matroid $D=(E,{\mathcal{F}})$, and let $A\subseteq E$. Define $(\boldsymbol\alpha, \boldsymbol\beta, \boldsymbol\gamma)\ast A$ to be the  weight system
 $\{(\alpha_e, \beta_e, \gamma_e): e\in E\ba A \} \cup   \{( \beta_e,\alpha_e, \gamma_e): e\in A\}$.  Also define, $(\boldsymbol\alpha, \boldsymbol\beta, \boldsymbol\gamma)+ A$ to be the  weight system  $\{(\alpha_e, \beta_e, \gamma_e): e\in E\ba A\} \cup   \{( \alpha_e, \gamma_e,\beta_e) : e\in A\}$.  For each word $w=w_1\ldots w_n \in  \langle * , +  \mid *^2, +^2, (*+)^3 \rangle$ define
\[(\boldsymbol\alpha, \boldsymbol\beta, \boldsymbol\gamma) w (A):= (\boldsymbol\alpha, \boldsymbol\beta, \boldsymbol\gamma) w_1 (A) w_{2}(A)\cdots w_n(A)  .\]
\begin{theorem}[(Brijder and Hoogeboom~\cite{BHpre2})]
\label{t.tdrandual}
 Let $D=(E,{\mathcal{F}})$ be a vf-safe delta-matroid.
If    $A_1, \ldots ,A_n \subseteq E$, and $g_1,\ldots, g_n \in \langle \ast ,+ \mid \ast^2, +^2, (\ast+)^3 \rangle$, and $\Gamma=  g_1(A_1) g_2(A_2)\cdots g_n(A_n)$, then
\[
Q_{(\boldsymbol\alpha, \boldsymbol\beta, \boldsymbol\gamma)}(D; t)=Q_{(\boldsymbol\alpha, \boldsymbol\beta, \boldsymbol\gamma)\Gamma}((D)\Gamma; t).
\]
 \end{theorem}

Once again we see that a delta-matroid polynomial is compatible with a ribbon graph polynomial.
\begin{proposition}\label{prop.qd}
Let $G=(V,E)$ be a  ribbon graph and $D(G)$ be its ribbon-graphic delta-matroid. Then
   \[Q(G; (\boldsymbol\alpha, \boldsymbol\beta, \boldsymbol\gamma),t)=  t^{k(G)} Q_{(\boldsymbol\beta, \boldsymbol\alpha, \boldsymbol\gamma)}(D(G) ;  t).\]
\end{proposition}
\begin{proof}
By Theorem~\ref{t.tdrandual},
$t^{k(G)} Q_{(\boldsymbol\beta, \boldsymbol\alpha, \boldsymbol\gamma)}(D(G) ;  t) =
t^{k(G)} Q_{(\boldsymbol\alpha, \boldsymbol\beta, \boldsymbol\gamma)}(D(G)\ast E) ;  t) $.
Then, by comparing Equation~\eqref{topotransdef} and Equation~\eqref{deltatransdef} we see that to prove the theorem it is to sufficient to prove that
$  d_{D(G)\ast E\ast B\bar{\ast}C} + k(G) = f(G^{\tau(C)} \ba B)    $ for all disjoint subsets $B$ and $C$ of $E(G)$. Theorem~\ref{t.td} and the properties of twisted duals from Section~\ref{4th} give
\[D(G)\ast E\ast B\bar{\ast}C
 = D(G^{\delta(E) \delta(B) \delta \tau\delta(C)})
= D(G^{\delta(A)\tau\delta(C)}),
\]
where $A=E-(B\cup C)$.
Then using that $r(D(G)_{\min})= v(G)-k(G)$, the  properties of twisted duals once again and that contracting an edge from a ribbon graph does not change its number of boundary components
\begin{align*}
d_{D(G)\ast E\ast B\bar{\ast}C} &= r((D(G)\ast E\ast B\bar{\ast}C)_{\min})  = r(D(  G^{\delta(A)\tau\delta(C)} )_{\min})
\\
&= v(  G^{\delta(A)\tau\delta(C)} ) - k( G^{\delta(A)\tau\delta(C)} )
=f(  G^{\delta(A)\tau\delta(C) \delta(E) } ) - k( G)
\\
&= f(G^{\tau(C) \delta(B) })- k( G)= f(G^{\tau(C)\delta(B)}/B)- k(G)= f(G^{\tau(C)} \ba B )- k( G),
 \end{align*}
as required.
\end{proof}

We note that the twisted duality relation for the ribbon graph version of the transition polynomial from~\cite{EMM} can be recovered from Proposition~\ref{prop.qd} and Theorem~\ref{t.tdrandual}.

In~\cite{BHpre2}, Brijder and Hoogeboom give a recursion relation (and a recursive definition) for the transition polynomial of a vf-safe delta-matroid, which we now reformulate.
\begin{theorem}[(Brijder and Hoogeboom~\cite{BHpre2})]\label{t.dqdct}
Let $D=(E,{\mathcal{F}})$ be a vf-safe delta-matroid and let $e\in E$.
\begin{enumerate}
 \item If $e$ is a  loop, then
\[  Q_{(\boldsymbol\alpha, \boldsymbol\beta, \boldsymbol\gamma)}(D;t) =
\alpha_e Q_{(\boldsymbol\alpha, \boldsymbol\beta, \boldsymbol\gamma)}(D\backslash e;t)
+ t \beta_e Q_{(\boldsymbol\alpha, \boldsymbol\beta, \boldsymbol\gamma)}(D/e;t)
+\gamma_e Q_{(\boldsymbol\alpha, \boldsymbol\beta, \boldsymbol\gamma)}((D+e)/e ;t) .\]

\item If $e$ is a trivial non-orientable ribbon loop, then
 \[  Q_{(\boldsymbol\alpha, \boldsymbol\beta, \boldsymbol\gamma)}(D;t) =
\alpha_e Q_{(\boldsymbol\alpha, \boldsymbol\beta, \boldsymbol\gamma)}(D\backslash e;t)
+ \beta_e Q_{(\boldsymbol\alpha, \boldsymbol\beta, \boldsymbol\gamma)}(D/e;t)
+t\gamma_e Q_{(\boldsymbol\alpha, \boldsymbol\beta, \boldsymbol\gamma)}((D+e)/e ;t) .\]

\item If $e$ is  a coloop, then
\[  Q_{(\boldsymbol\alpha, \boldsymbol\beta, \boldsymbol\gamma)}(D;t) =
t \alpha_e Q_{(\boldsymbol\alpha, \boldsymbol\beta, \boldsymbol\gamma)}(D\backslash e;t)
+ \beta_e Q_{(\boldsymbol\alpha, \boldsymbol\beta, \boldsymbol\gamma)}(D/e;t)
+\gamma_e Q_{(\boldsymbol\alpha, \boldsymbol\beta, \boldsymbol\gamma)}((D+e)/e ;t) .\]

\item If $e$  does not meet the above conditions, then
 \[  Q_{(\boldsymbol\alpha, \boldsymbol\beta, \boldsymbol\gamma)}(D;t) =
\alpha_e Q_{(\boldsymbol\alpha, \boldsymbol\beta, \boldsymbol\gamma)}(D\backslash e;t)
+ \beta_e Q_{(\boldsymbol\alpha, \boldsymbol\beta, \boldsymbol\gamma)}(D/e;t)
+\gamma_e Q_{(\boldsymbol\alpha, \boldsymbol\beta, \boldsymbol\gamma)}((D+e)/e ;t) .\]

\item $Q_{(\boldsymbol\alpha, \boldsymbol\beta, \boldsymbol\gamma)}(D;t) = 1$, when $E=\emptyset$.
\end{enumerate}

\end{theorem}
The theorem is Theorem~3 of~\cite{BHpre2} except that $Q_{(\boldsymbol\alpha, \boldsymbol\beta, \boldsymbol\gamma)}((D\bar{\ast}e)\ba e ;t)$ appears in the relations in~\cite{BHpre2} rather than $Q_{(\boldsymbol\alpha, \boldsymbol\beta, \boldsymbol\gamma)}((D+e)/e ;t)$. As we noted after Proposition~\ref{p.pdc}, these two delta-matroids are the same.

By using Proposition~\ref{prop.qd} and considering delta-matroids instead of ribbon graphs, we can see that Theorem~\ref{t.dqdct} is the direct ribbon graph analogue of the  recursion relation for the ribbon graph version of the transition polynomial from~\cite{EMM}: \begin{equation*}\label{e.dctrans}
Q(G; (\boldsymbol\alpha, \boldsymbol\beta, \boldsymbol\gamma), t)= \alpha_e Q(G/e; (\boldsymbol\alpha, \boldsymbol\beta, \boldsymbol\gamma), t) + \beta_e Q(G\backslash e; (\boldsymbol\alpha, \boldsymbol\beta, \boldsymbol\gamma), t)+\gamma_e Q(G^{\tau(e)}/e; (\boldsymbol\alpha, \boldsymbol\beta, \boldsymbol\gamma), t).
\end{equation*}

\medskip

The \emph{Bollob\'as--Riordan polynomial} of a ribbon graph $G$, from~\cite{BR1,BR2}, is given by
\[R(G;x,y,z) = \sum_{A \subseteq E( G)}   (x - 1)^{r( E ) - r( A )}   y^{|A|-r(A)} z^{\gamma(A)}  ,\]
where $r(A)$  and  $\gamma(A)$ are as defined in  Section~\ref{dbh}. Chun et al. established in~\cite{CMNR} that  $R(G)$ is delta-matroidal in the sense that it is determined by $D(G)$ (this fact also follows from Traldi~\cite{Tr15trans}).
For a delta-matroid $D$ and subset $A$ of its elements, let $D|A:= D\setminus A^c$.
In~\cite{CMNR},  the Bollob\'as--Riordan polynomial was extended to delta-matroids by
  \[   R(D;x,y,z) := \sum_{A \subseteq E}   (x - 1)^{r_{D_{\min}}( E ) - r_{D_{\min}}( A )}   y^{|A|-r_{D_{\min}}(A)} z^{r((D|A)_{\max})-r((D|A)_{\min})},\]
and it was shown that $R(G;x,y,z)=R(D(G);x,y,z)$, for each ribbon graph $G$.

It is known that for ribbon graphs, the Penrose polynomial and the  Bollob\'as--Riordan polynomial are encapsulated by the transition polynomial (see~\cite{EMM11b} and~\cite{ES11} respectively). These relations between the polynomials hold more generally for delta-matroids.
Brijder and Hoogeboom~\cite{BH12} showed that for a vf-safe delta-matroid, $Q_{(0,1,-1)}(D; \lambda) = P(D; \lambda)$. We also obtain the Bollob\'as--Riordan polynomial as a specialization of the transition polynomial.
\begin{proposition}\label{p.dqpr}
Let $D$ be a delta-matroid. Then
\[Q_{(1,\sqrt{y/x}, 0)} \big(D ; \sqrt{xy}\big) =   (\sqrt{y/x})^{r(E)}   R(D;  x+1, y, 1/\sqrt{xy}).\]
\end{proposition}
\begin{proof}
Let $D=(E,\mathcal{F})$ be a delta-matroid.
For simplicity we write $ n(A)$ for  $n_{D_{\min}}(A)$, and $r(A)$ for  $r_{D_{\min}}(A)$.
In~\cite{CMNR}, it is shown that
$r((D|A)_{\min})=r(A)$ and
\[r((D|A)_{\max})=\rho_D(A)-n(E) + n(A),\]
where
\begin{align*}
 \rho_D (A)  &=  |E|-\min\{|A\bigtriangleup F| :  F\in \mathcal{F}(D)\} =  |E|-\min\{ |F'| :    F'=A\bigtriangleup F, F\in \mathcal{F}(D)\}\\
  &=  |E|-\min\{|F'| :    F'\in \mathcal{F}(D\ast A)\} = |E| - d_{D\ast A}.
 \end{align*}
Thus
\begin{align*}
 R(D;x+1,y,1/\sqrt{xy}) &= \sum_{B \subseteq E}   x^{r(E) - r(A)}   y^{|A|-r(A)}
 (1/\sqrt{xy})^{|E|-d_{D\ast B}-n(E)+n(A)-r(A)}\\
 & =  (\sqrt{x/y})^{r(E)}  \sum_{(A,B,C) \in \mathcal{P}_3  (E)}  1^{|A|} (\sqrt{y/x})^{|B|}  0^{|C|} (\sqrt{xy})^{d_{D\ast B}}
  \\
 & = (\sqrt{x/y})^{r(E)}  Q_{(1, \sqrt{y/x}, 0)}(D; \sqrt{xy}).
  \end{align*}
\end{proof}

Our final result is the following corollary.
\begin{corollary}
Let $D=(E,{\mathcal{F}})$ be a  delta-matroid, and $e\in E$.
\begin{enumerate}
 \item If $e$ is a  loop, then
\[     R(D;  x+1, y, 1/\sqrt{xy}) =
    R(D \ba e;  x+1, y, 1/\sqrt{xy})
+ y   R(D / e;  x+1, y, 1/\sqrt{xy}).
 \]

\item If $e$ is a non-trivial, orientable ribbon loop, then
 \[     R(D;  x+1, y, 1/\sqrt{xy}) =
   R(D \ba e;  x+1, y, 1/\sqrt{xy})
+ (y/x)    R(D / e;  x+1, y, 1/\sqrt{xy}).
 \]

 \item If $e$ is a  non-orientable ribbon loop, then
 \[     R(D;  x+1, y, 1/\sqrt{xy}) =
     R(D \ba e;  x+1, y, 1/\sqrt{xy})
+ (\sqrt{y/x})   R(D / e;  x+1, y, 1/\sqrt{xy}).
 \]

\item If $e$ is  a coloop, then
\[     R(D;  x+1, y, 1/\sqrt{xy}) =
x  R(D \ba e;  x+1, y, 1/\sqrt{xy})
+   R(D / e;  x+1, y, 1/\sqrt{xy}).
 \]

\item If $e$ is not a ribbon loop or a coloop, then
 \[     R(D;  x+1, y, 1/\sqrt{xy}) =
   R(D \ba e;  x+1, y, 1/\sqrt{xy})
+  R(D / e;  x+1, y, 1/\sqrt{xy}).
 \]

\item $R( D ;  x+1, y, 1/\sqrt{xy}   ) = 1$, when $E=\emptyset$.
\end{enumerate}
\end{corollary}
\begin{proof}
The result follows  from Theorem~\ref{t.dqdct} and Proposition~\ref{p.dqpr} after simplifying the terms $d_{D\ba e}-d_D$ and $d_{D/e}-d_D$. The simplification of these terms is straightforward except for the computation of $d_{D/e}-d_D$ when $e$ is a non-trivial orientable ribbon loop.
Since $e$ is an orientable ribbon loop, it is not a ribbon loop in $D*e$, so $d_{D*e}=d_D+1$. Moreover, $e$ is non-trivial, so it is not a coloop of $D*e$, therefore by Lemma~\ref{lem:useful}, there is a minimal feasible set of $D*e$ that does not contain $e$. Hence $d_{D/e}=d_{(D*e)\ba e}=d_{D*e}=d_D+1$.
\end{proof}
By restricting to the delta-matroids of ribbon graphs we can recover the deletion-contraction relations for the ribbon graph version of the polynomial that appears in Corollary~4.42 of~\cite{EMMbook}.

\bibliographystyle{amsplain}

\affiliationone{
   C. Chun \\
   Mathematics Department, \\United States Naval Academy\\Annapolis\\ Maryland \\21402-5002\\ United States of America
   \email{chun@usna.edu}}
   \affiliationtwo{
   I. Moffatt\\
   Department of Mathematics\\ Royal Holloway \\ University of London\\ Egham Surrey \\TW20 0EX\\ United Kingdom
   \email{iain.moffatt@rhul.ac.uk}}
  \affiliationthree{ S. D. Noble\\
  Department of Economics, Mathematics and Statistics\\ Birkbeck, University of London\\ London\\ WC1E 7HX\\ United Kingdom
   \email{s.noble@bbk.ac.uk}}
\affiliationfour{
  R. Rueckriemen\\
   Aschaffenburger Strasse 23 \\ 10779 Berlin
   \email{ralf@rueckriemen.de}}
\end{document}